\documentclass[a4paper]{amsart}

\usepackage[shortlabels]{enumitem}
\setlist[enumerate]{label={\arabic*.}}
\setlist[description]{font=\normalfont\slshape}

\usepackage[dvipsnames]{xcolor}
\usepackage[backref]{hyperref}
\hypersetup{colorlinks=true,linkcolor=Maroon,citecolor=OliveGreen}

\usepackage{mathtools}
\mathtoolsset{showonlyrefs}

\usepackage{microtype}

\newtheorem{theorem}{Theorem}[section]
\newtheorem{lemma}[theorem]{Lemma}
\newtheorem{proposition}[theorem]{Proposition}
\newtheorem{corollary}[theorem]{Corollary}

\theoremstyle{definition}

\newtheorem{remark}[theorem]{Remark}

\usepackage{amssymb}
\usepackage{mathtools}
\renewcommand\bar{\overline}
\renewcommand\hat{\widehat}
\renewcommand\tilde{\widetilde}
\newcommand{\eps}{\epsilon}

\newcommand\opr[1]{\operatorname{#1}}

\def\C{\mathbf{C}}

\def\F{\mathbf{F}}

\def\Irr{\opr{Irr}}

\def\diam{\opr{diam}}
\def\Cay{\opr{Cay}}
\def\Sch{\opr{Sch}}

\def\Cl{\opr{Cl}}
\def\GCl{\opr{GCl}}
\def\SCl{\opr{SCl}}
\def\PSCl{\opr{PSCl}}

\def\GL{\opr{GL}}
\def\SL{\opr{SL}}
\def\PSL{\opr{PSL}}
\def\GU{\opr{GU}}
\def\SU{\opr{SU}}
\def\Sp{\opr{Sp}}
\def\GO{\opr{GO}}
\def\SO{\opr{SO}}

\def\M{\opr{M}}

\def\P{\mathbf{P}}
\def\E{\mathbf{E}}
\def\Var{\opr{Var}}
\def\supp{\opr{supp}}  
\def\fix{\opr{fix}}  
\def\codim{\opr{codim}}
\def\sp{\opr{span}}
\def\tr{\opr{tr}}
\def\rank{\opr{rank}}
\def\pp{\mathcal{P}} 
\def\CC{\mathfrak{C}} 
\def\XX{\mathcal{X}} 
\def\sgn{\opr{sgn}} 

\def\adjcy{\mathcal{A}}  
\def\MM{\mathfrak{M}}
\def\irrpolys{\mathfrak{I}}

\def\alphabet{\{\xi_1, \dots, \xi_k\}} 
\def\signedalphabet{\{\xi_1^{\pm1}, \dots, \xi_k^{\pm1}\}}
\def\barF{\overline{\F_q}}

\newcommand\floor[1]{\left\lfloor{#1}\right\rfloor}
\newcommand\ceil[1]{\left\lceil{#1}\right\rceil}
\newcommand\pfrac[2]{\left(\frac{#1}{#2}\right)}
\newcommand\br[1]{\left(#1\right)}

\def\an{\textup{an}}
\def\ab{\textup{ab}}

\usepackage{todonotes}

\renewcommand\subset{\subseteq}

\begin{document}
\baselineskip=13pt 

\title[Babai's conjecture for random generators]{Babai's conjecture for high-rank classical groups with random generators}

\author{Sean Eberhard}
\address{Sean Eberhard, Centre for Mathematical Sciences, Wilberforce Road, Cambridge CB3~0WB, UK}
\email{eberhard@maths.cam.ac.uk}

\author{Urban Jezernik}
\address{Urban Jezernik, Alfr\'ed R\'enyi Institute of Mathematics, Hungarian Academy of Sciences, Re\'altanoda utca 13-15, H-1053, Budapest, Hungary}
\email{jezernik.urban@renyi.hu}

\thanks{S. Eberhard has received funding from the European Research Council (ERC) under the European Union’s Horizon 2020 research and innovation programme (grant agreement No. 803711).
U. Jezernik has received funding from the European
Research Council (ERC) under the European Union’s Horizon 2020 research and innovation programme (grant agreement No. 741420).}

\begin{abstract}
Let $G = \SCl_n(q)$ be a quasisimple classical group with $n$ large, and let $x_1, \dots, x_k \in G$ be random, where $k \geq q^C$.
We show that the diameter of the resulting Cayley graph is bounded by $q^2 n^{O(1)}$ with probability $1 - o(1)$.
In the particular case $G = \SL_n(p)$ with $p$ a prime of bounded size, we show that the same holds for $k = 3$.
\end{abstract}

\maketitle

\setcounter{tocdepth}{1}
\tableofcontents

\section{Introduction}

Let $G$ be a group and $S$ a symmetric ($S = S^{-1}$) subset of $G$.
Write $\Cay(G, S)$ for the associated Cayley graph: the graph whose vertices are the elements $g \in G$ and whose edges are pairs $\{g, sg\}$ with $g\in G, s\in S$.
The graph $\Cay(G, S)$ is connected if and only if $S$ generates $G$, and its diameter is equal to the smallest $d$ such that $(S \cup \{1\})^d = G$.
A well-known conjecture of Babai~\cite{babai1992diameter} states that
\[
  \diam\Cay(G, S) = (\log |G|)^{O(1)},
\]
uniformly over all nonabelian finite simple groups $G$ and symmetric generating sets $S$. In other words, every connected Cayley graph of a nonabelian finite simple group has diameter within a power of the trivial lower bound.

By the classification of finite simple groups, Babai's conjecture splits into essentially three broad cases:
\begin{enumerate}
  \item groups of Lie type of bounded rank over $\F_q$ with $q \to \infty$;
  \item classical groups of unbounded rank over $\F_q$ with $q$ arbitrary;
  \item alternating groups $A_n$ with $n \to \infty$.
\end{enumerate}

For groups of Lie type and bounded rank, Babai's conjecture is now completely resolved, following breakthrough work of Helfgott~\cite{helfgott2008growth}, Pyber--Szab{\'o}~\cite{pyber2016growth}, and Breuillard--Green--Tao~\cite{breuillard2011approximate}.
In the other two cases the conjecture remains open.
For the alternating groups, Helfgott and Seress~\cite{helfgott2014diameter} proved that
\[
  \diam\Cay(A_n, S) = \exp O((\log n)^4 \log \log n).
\]
For comparison, Babai's conjecture (folkloric in this case) asserts that
\[
  \diam\Cay(A_n, S) = n^{O(1)};
\]
thus we have a quasipolynomial bound instead of the expected polynomial bound.
The case of classical groups of unbounded rank on the other hand is still wide open.
The best bounds currently known are due to Biswas--Yang and Halasi--Mar\'oti--Pyber--Qiao:
\begin{align}
  \diam\Cay(G, S) &\leq q^{O(n (\log n + \log q)^3)}&& \text{(\cite{biswas--yang})}; \\
  \diam\Cay(G, S) &\leq q^{O(n (\log n)^2)} && \text{(\cite{HMPQ})}. \label{HMPQ}
\end{align}
By contrast, Babai's conjecture in this case asserts that
\[
  \diam\Cay(G, S) \leq (n \log q)^{O(1)},
\]
so we are still exponentially stupid.
A key open case is the family of groups $\SL_n(2)$ with $n$ tending to infinity.

In all cases, an important subproblem is the case of random generators (see, e.g., \cite[Problem~10.8.6]{lubotzky2010discrete}). Let $k \geq 2$ be a small constant and let $S = \{x_1^{\pm1}, \dots, x_k^{\pm1}\}$, where $x_1, \dots, x_k \in G$ are uniform and independent.
For groups of Lie type of bounded rank, it was proved by Breuillard, Green, Guralnick, and Tao~\cite{BGGT} that $\Cay(G, S)$ is almost surely\footnote{Throughout the paper, we use the terms ``almost surely'' or ``with high probability'' to mean with probability $1-o(1)$ as the relevant parameters tend to infinity.} an expander, and in particular
\[
  \diam \Cay(G, S) = O(\log |G|).
\]
There is no consensus about whether such a strong bound is likely to hold for groups of unbounded rank. Babai's conjecture for $A_n$ and random generators was an open problem for some time. The first polynomial bound was proved by Babai and Hayes, and the exponent has been lowered by Schlage-Puchta and Helfgott--Seress--Zuk:
\begin{align}
  \diam\Cay(A_n, S) &\leq n^{7+o(1)} && \text{(\cite{babai--hayes})}; \\
  \diam\Cay(A_n, S) &\leq O(n^3 \log n) && \text{(\cite{schlage-puchta})}; \\
  \diam\Cay(A_n, S) &\leq n^2 (\log n)^{O(1)} && \text{(\cite{HSZ})}. \label{HSZ}
\end{align}

In this paper we consider the case of high-rank classical groups over a small field.
Recall that these are obtained from the groups
\begin{equation}
  \begin{array}{llll}
  \GL_n(q),
  &\Sp_n(q),
  &\GO_{n}^{(\pm)}(q),
  &\GU_n(q),
  \end{array}
\end{equation}
of automorphisms of a finite vector space $V = \F_q^n$, in the latter three cases equipped with a nondegenerate alternating, quadratic, or hermitian form, respectively.
Throughout we write $\GCl_n(q)$ for any of these groups, and $\SCl_n(q)$ for the corresponding derived subgroup
\begin{equation}
  \begin{array}{llll}
  \SL_n(q),
  &\Sp_n(q),
  &\Omega_{n}^{(\pm)}(q),
  &\SU_n(q).
  \end{array}
\end{equation}
We will write $\Cl_n(q)$ for any intermediate group:
\[
  \SCl_n(q) \leq \Cl_n(q) \leq \GCl_n(q).
\]
Omitting a few small exceptional cases, $\SCl_n(q)$ is a quasisimple group, so Babai's conjecture applies.\footnote{The diameter of $\SCl_n(q)$ with respect to a set $S$ is essentially the same (up to a factor of $3$) as the diameter of the simple quotient $\PSCl_n(q)$ with respect to $S \bmod Z$. Indeed, if $S^d = G$ then certainly $S^d Z = G$, and conversely if $S^d Z = G$ then it is possible to show that $S^{3d} = G$. Hence there is no need to consider $\PSCl_n(q)$ explicitly.}
For $\SCl_n(q)$ with $n$ large and random generators, the best bound out there is just the uniform bound \eqref{HMPQ}.

There is a promising programme of Pyber, which aims to prove Babai's conjecture in three steps.
The programme is motivated by the positive solution in the case of random generators in alternating groups, especially the result of Babai--Beals--Seress~\cite{BBS} that $\diam \Cay(A_n, S) \leq n^{O(1)}$ provided only that $S$ contains an element of degree at most $n/(3 + \eps)$.
Here the \emph{degree} of a permutation is the number of non-fixed points.
Analogously, the \emph{degree} of an element $g \in \GL_n(q)$ is defined to be the rank of $g - 1$, and Pyber's programme is the following.

\begin{enumerate}
  \item Given some generators, find an element whose degree is at most $(1-\eps)n$.
  \item Given an element of degree $(1-\eps)n$, find an element of minimal degree.
  \item Given an element whose degree is minimal, finish the proof.
\end{enumerate}

In the case of alternating groups, step 3 is essentially trivial, since there are only $O(n^3)$ $3$-cycles in $A_n$, but for $\SCl_n(q)$ it is highly nontrivial. In the case of $\SL_n(p)$, $p$ prime, step 3 was accomplished recently by Halasi~\cite{halasi}.

We have two things to contribute in the case of large $n$, small $q$.
First, assuming we have at least 3 random generators, we will do steps 1 and 2 of Pyber's programme.

\begin{theorem}
  \label{thm1:finding-a-transvection}
  Let $G = \Cl_n(q)$, and assume $\log q < c n / \log^2 n$ for a sufficiently small constant $c>0$. Let $x, y, z \in G$ be random. Then with probability $1 - e^{-cn}$ there is a word $w \in F_3$ of length $n^{O(\log q)}$ such that $w(x, y, z)$ has minimal degree in $G' = \SCl_n(q)$.
\end{theorem}

Combined with Halasi's result, this settles Babai's conjecture for $\SL_n(p)$, $p$ prime and bounded, with at least 3 random generators.

\begin{theorem}
  \label{diameter-thm-1}
Let $\SL_n(p) \leq G \leq \GL_n(p)$, where $p$ is prime and $\log p < cn / \log^2 n$.
Let $x,y,z$ be elements of $G$ chosen uniformly at random,
and let $S = \{ x^{\pm1}, y^{\pm1}, z^{\pm1} \}$.
Then with probability $1 - e^{-cn}$ we have
\begin{align}
  &\langle S \rangle \geq \SL_n(p),~\text{and} \\
  &\diam \Cay(\langle S \rangle, S) \leq n^{O(\log p)}.
\end{align}
\end{theorem}

Second, assuming we have sufficiently many random generators depending on $q$, we will do step 3 in a particularly satisfactory way. In fact, we will prove that the Schreier graph of the action of $G$ on $O(1)$-tuples of vectors is almost surely a union of expander graphs. (The analogous result for the symmetric group is a result of Friedman, Joux, Roichman, Stern, and Tillich~\cite{FJRST}, and was essential in \cite{HSZ}.)

\begin{theorem} \label{thm2:schreier-expansion}
Let $G = \Cl_n(q)$, and let $x_1, \dots, x_k \in G$ be random.
Let $W$ be the set of $r$-tuples of vectors in the natural module $V = \F_q^n$.
Assume that $r < cn^{1/3}$, and that $k \geq q^{C r^3}$.
Then almost surely the Schreier graph of $G$ generated by $x_1, \dots, x_k$ on any of its orbits in $W$ has a uniform spectral gap.
\end{theorem}

As we will explain, this implies that if we have an element of minimal degree then by conjugation we can rapidly obtain a full conjugacy class of elements of minimal degree, and it follows in short order that the diameter of $G$ is not too large. This completes the proof of Babai's conjecture for $\SCl_n(q)$ for $k$ random generators, as long as $k$ is sufficiently large compared to $q$.

\begin{theorem}
\label{diameter-thm-2}
There are constants $c, C>0$ so that the following holds.
Let $G = \Cl_n(q)$, where $n > C$.
Let $x_1, \dots, x_k$ be elements of $G$ chosen uniformly at random,
where $k > q^C$,
and let $S = \{ x_1^{\pm1}, \dots, x_k^{\pm1}\}$.
Then with probability $1 - q^{-cn}$ we have
\begin{align}
  &\langle S \rangle \geq \SCl_n(q),~\text{and} \\
  &\diam \Cay(\langle S \rangle, S) \leq q^2 n^C.
\end{align}
\end{theorem}

\begin{corollary} \label{cor:babai}
Babai's conjecture holds in the following two cases:
\begin{enumerate}[(1)]
  \item $\SL_n(p)$, $p$ prime and bounded, and at least 3 random generators;
  \item $\SCl_n(q)$ and at least $q^C$ random generators, where $C$ is an absolute constant.
\end{enumerate}
\end{corollary}

Our method does not depend on the classification of finite simple groups (CFSG) in any way.
Having a CFSG-free method is valuable for transparency,
but moreover we think it is essential for attacking Babai's conjecture.
It is well-known that two random elements of $\SCl_n(q)$ almost surely generate the group: this is a result of Kantor and Lubotzky~\cite{kantor--lubotzky}. Kantor and Lubotzky rely on CFSG through Aschbacher's theorem, so unfortunately their method does not adapt well to proving diameter bounds.
By contrast, in \cite{EV-SL} the first author and Virchow found a CFSG-free proof in the case of $\SL_n(q)$ and expressed the hope that the method would be generalizable. We recycle several ideas from that paper in the present one.

Perhaps the most important idea in our method is the idea that if $x, y, z \in G$ are random and independent, then the elements $x w(y, z)$ for all short words $w \in F_2$ behave roughly independently, which allows us to imitate having many more than just $3$ generators. This is a more powerful version of the ``$xy^i$ trick'', which comes originally from \cite[Section~4]{BBS} and has been essential in all subsequent work on the random generator subproblem in high rank.

Let us mention one further result, of independent interest. In the appendix we give analogous arguments for $A_n$, based on the standard fanciful idea that $A_n = \PSL_n(1)$. The value of doing so is mostly motivational, but we also obtain a new result. Provided $k \geq 3$, we sharpen \eqref{HSZ} to
\[
  \diam \Cay(A_n, S) \leq O(n^2 \log n).
\]
This is a modest improvement, but it is interesting for being conjecturally sharp for any proof which uses elements of small support as a stepping stone. Decreasing the exponent $2$ appears to require a radically new idea.

\subsection*{Reader's guide}

We first record some preliminaries (Section~\ref{sec:preliminaries})
regarding asymptotic notation,
Cayley and Schreier graphs,
classical groups and their associated formed spaces and
the notions of degree and support, and adjacency operators.

Next we turn to a more specialized preparatory section
(Section~\ref{sec:word_maps}) dealing with
word maps, where we introduce the vocabulary of queries, coincidences, and trajectories.
Briefly, the idea is that if $w \in F_k$ is a given word, $v \in V$ a given vector, and $x_1, x_2, \dots, x_k \in G$ random, then evaluating $w(x_1, \dots, x_k) v$ can be thought of as a kind of random walk.
As much as possible we recycle the key language
used by \cite{FJRST}
in the case of the symmetric group.
The tools of this section will be used in two essentially different ways in the rest of the paper.

We proceed (Section~\ref{sec:prob-small-support}) by showing that
a given short word $w$ evaluated at random elements $x_1, \dots, x_k \in G$ almost surely has large support (Theorem~\ref{theorem:supp0.99}).
This is a kind of antithesis to step 1 of Pyber's programme: all sufficiently short words in random generators will in fact \emph{fail} to have degree $(1-\eps)n$.
However, this is interesting when combined with recent character bounds of Guralnick--Larsen--Tiep~\cite{GLT1,GLT2}, as it implies that the character ratio $\chi(w(x_1, \dots, x_k)) / \chi(1)$ is almost surely small for each nonlinear character $\chi$ (Corollary~\ref{cor:expected_character_bound_exponential}).

This bound on the expectation of $\chi(w(x_1, \dots, x_k)) / \chi(1)$ is one of the two main ingredients in the ``$x w(y, z)$ trick'', which is the subject of Section~\ref{sec:reaching_normal_subset}. This trick shows that, given random generators $x_0, x_1, \dots, x_k$, one can almost surely find a short word $x_0 w(x_1, \dots, x_k)$ lying in a given normal subset $\CC \subset G$, provided that the density of $\CC$ is large compared to the expected values of character ratios.
The trick is a simple consequence of the second moment method, following the observation that the elements $x_0 w(x_1, \dots, x_k)$ for various $w$ are approximately pairwise independent.

The other main ingredient is the construction of an appropriate normal set $\CC$. This is the subject of Section~\ref{sec:obtaining_minimal_support}. For each classical group we find a large normal set $\CC$, all of whose fibres over $G^\ab$ are large (allowing us to ignore linear characters), and a small integer $m$ such that for every $g \in \CC$ the power $g^m$ has minimal degree in $\SCl_n(q)$. This completes the proof of Theorem~\ref{thm1:finding-a-transvection}.

Once we have an element of minimal degree, we can act on that element by conjugation. Since the minimal degree in all cases is at most $2$, this action is a constituent of the usual permutation action on $4$-tuples of vectors. We analyze this action by again using the language of trajectories and coincidences, and the \emph{trace method}: we bound a high moment of the second eigenvalue by bounding the trace of the corresponding power of the adjacency matrix, interpretting the latter in terms of closed trajectories.
This is analogous to a result for the symmetric group due to Friedman, Joux, Roichman, Stern, and Tillich~\cite{FJRST}, building on earlier work of Broder--Shamir~\cite{broder-shamir}.
However, in the case of classical groups there are some extra combinatorial complications that do not arise for symmetric groups.

We first focus (Section~\ref{sec:closed_trajectories_one_coincidence}) on describing the structure of a closed trajectory with only one coincidence. We deal with the motivational case of $G$ acting on $V$ first, and then generalize to the action on tuples of vectors.

These results are then (Section~\ref{sec:trace-method}) used
to show that, in an orbit of $G$ of size $N$, the probability
that a trajectory closes is close to $1/N$, with a small relative error.
Again we first deal with the motivational case of $G$ acting on $V$.
Provided that we have sufficiently many generators in terms of $q$, these bounds are good enough for the trace method to work. This completes the proof of Theorem~\ref{thm2:schreier-expansion}.

Finally, in Section~\ref{sec:diameter} we collect results and deduce Theorems~\ref{diameter-thm-1} and \ref{diameter-thm-2}.

Many (but not all) of our arguments have natural analogues for the symmetric group. For independent interest and for motivation, these are presented in Appendix~\ref{appendix:Sn}.

\subsection*{Acknowledgments}

We thank L\'aszl\'o Pyber, Endre Szab\'o, and P\'eter Varj\'u for
helpful discussions.
Thanks are due to Emmanuel Breuillard and Bob Guralnick for discussions pertaining to the low-degree representation theory of $\SCl_n(q)$, and to Aner Shalev for discussions about character bounds.
We thank Zolt\'an Halasi for sharing the preprint \cite{halasi}.
We would also like to thank two anonymous referees for a thorough inspection of the paper and suggesting many improvements.

\section{Preliminaries}
\label{sec:preliminaries}

This section fixes some notation and definitions that will be relevant throughout the paper.
The reader needing an introduction to expansion, particularly in Cayley and Schreier graphs, could consult Kowalski~\cite{kowalski2019introduction}.
For an introduction to classical groups, see Aschbacher~\cite[Chapter~7]{aschbacher} or Grove~\cite{grove2002classical}.

\subsection{Asymptotic notation}

Many of the arguments we will use are of asymptotic nature
and we adopt standard asymptotic notation to state these.
Given functions $f,g$, we write $f \ll g$ or equivalently
$f = O(g)$ to denote that there are absolute constants
$N, C > 0$ so that $|f(n)| \leq C \cdot g(n)$ for all
$n \geq N$.
Let $f \asymp g$ mean that $f \ll g$ and $g \ll f$.
We write $f = o(g)$ to denote that for every $\eps > 0$
there is a constant $N$ so that $|f(n)| \leq \eps \cdot g(n)$
for all $n \geq N$.
Let $f = \omega(g)$ mean that $g = o(f)$.

We will generally write statements that involve anonymous (usually absolute)
constants by using $c$ for small constants and $C$ for big constants.

\subsection{Cayley and Schreier graphs}

Let $G$ be a group with generating set $S$ satisfying $S = S^{-1}$. 
The (undirected, left) \emph{Cayley graph} $\Cay(G,S)$ is 
the graph whose vertices are elements of $G$ and whose edges are pairs $\{ g, s g \}$ for $g \in G, s \in S$.

More generally, the (undirected) \emph{Schreier graph} $\Sch(G,S,\Omega)$ associated to a transitive action of $G$ on a set $\Omega$ is 
the graph whose vertices are elements of $\Omega$ and whose edges are pairs $\{ \omega, s \omega \}$ for $\omega \in \Omega, s \in S$.
Cayley graphs are Schreier graphs for the left regular representation of $G$ on itself.

Let $\Gamma$ be a connected graph. One can view $\Gamma$ as a metric space in the following way. Define the length of a path in $\Gamma$ to be the number of edges on the path, and let the distance $d_\Gamma(v_1, v_2)$ between any two vertices 
$v_1, v_2 \in V(\Gamma)$ be the length of the shortest path between $v_1, v_2$. The \emph{diameter} of a graph $\Gamma$ is
\[
\diam \Gamma = \max_{v_1, v_2 \in V(\Gamma)} d_{\Gamma}(v_1, v_2).
\]

The diameter of $\Cay(G, S)$ is just the smallest $d \geq 0$ such that $(S \cup \{1\})^d = G$.

\subsection{Classical groups}

Throughout the paper we write $\SCl_n(q) \leq \GCl_n(q) \leq \GL_n(q)$ for any of the following groups:
\[
\begin{array}{rlllll}
\GCl_n(q):&
    &\GL_n(q),
    &\Sp_n(q),
    &\GO_{n}^{(\pm)}(q),
    &\GU_n(q), \\[0.2em]
\SCl_n(q):&
    &\SL_n(q),
    &\Sp_n(q),
    &\Omega_{n}^{(\pm)}(q),
    &\SU_n(q).
\end{array}
\]
In all cases the defining module is $V = \F_q^n$. We sometimes refer to the first case as the \emph{linear} case. We make the following conventions in the other cases (notation in other literature sometimes differs, particular in the $\GU$ case):
\begin{description}
  \item[$\Sp_n$] $n$ must be even.
  \item[$\GO_n^{(\pm)}$] $\Omega_n(q) = \SO_n(q)'$. If $n$ is even there are two possibilities, denoted $\GO_n^+(q)$ and $\GO_n^-(q)$, depending on the choice of quadratic form. If $n$ is odd there is only $\GO_n(q)$, and $q$ must be odd.
  \item[$\GU_n$] $q$ must be a square $q_0^2$. The field automorphism of $\F_q$ of order $2$ is denoted $\theta$.
\end{description}

We write $\Cl_n(q)$ for any intermediate group:
\[
  \SCl_n(q) \leq \Cl_n(q) \leq \GCl_n(q).
\]
Note that any such group corresponds to a subgroup of the abelianization $\GCl_n(q)^\ab$, which is given as follows:
\begin{align}
  \GL_n(q)^\ab &\cong \F_q^\times,\\
  \Sp_n(q)^\ab &\cong 1,\\
  \GO_n^{(\pm)}(q)^\ab &\cong C_2 \times C_2 & (q~\text{odd}, n\geq 2),\\
  \GO_n^{\pm}(q)^\ab &\cong C_2 & (q~\text{even}, n~\text{even}),\\
  \GU_n(q)^\ab &\cong \{u \in \F_q : u u^\theta = 1\}.
\end{align}

\subsection{Binary and quadratic forms}

In all cases we write $f$ for the defining invariant binary form; thus $f$ is zero in the linear case, alternating in the symplectic case, symmetric in the orthogonal case, and hermitian in the unitary case.
Except in the linear case, $f$ is nondegenerate.

In the orthogonal case, we write $Q$ for the relevant quadratic form. Recall that $Q$ is related to $f$ by
\begin{equation} \label{eq:Qf-polarization}
  Q(u + v) = Q(u) + Q(v) + f(u, v);
\end{equation}
in particular, in odd characteristic,
\[
  Q(v) = f(v, v) / 2.
\]
In even characteristic, $Q$ is not determined by $f$, but is part of the defining data (and $f$ is determined by $Q$ via \eqref{eq:Qf-polarization}).
In the unitary case we write $Q$ for the function
\[
  Q(v) = f(v, v),
\]
which we may regard as a quadratic form over $\F_{q_0}$.
In the other cases define $Q \equiv 0$. Define also $q_0 = q$ in the orthogonal case and $q_0=1$ in the linear and symplectic cases, so that $Q$ always takes values in a $q_0$-element space.

It is important that we are able to count solutions to $Q(v) = x$ in any affine subspace.

\begin{lemma} \label{lem:subspace-counting-lemma}
Let $v_0 + W$ be an affine subspace of $V$ of codimension $s$.
The number of $v \in v_0 + W$ with a specified value of $Q(v)$ is within $q^{n - s}/q_0 \pm q^{n/2}$.
\end{lemma}
\begin{proof}(Cf.~Dickson~\cite[Chapter~IV]{dickson}.)
This is trivial in the linear and symplectic cases: $Q \equiv 0$, so the number is exactly $q^{n-s}$. The unitary case reduces to the orthogonal case by restriction of scalars, so it suffices to consider the orthogonal case.

For $x \in \F_q$, let
\[
  \Phi(x) = |\{v \in v_0 + W : Q(v) = x\}|.
\]
The Fourier transform of $\Phi$ is
\begin{align*}
  \hat{\Phi}(\chi)
  &= \sum_{x \in \F_q} \Phi(x) \bar{\chi(x)} & (\chi \in \hat{\F_q}) \\
  &= \sum_{w \in W} \chi(-Q(v_0 + w)).
\end{align*}
For nontrivial $\chi$ we have
\begin{align*}
  |\hat{\Phi}(\chi)|^2
  &= \sum_{w, h \in W} \chi(-Q(v_0+w) + Q(v_0+w+h)) \\
  &= \sum_{w, h \in W} \chi(Q(h) + \Phi(v_0+w,h)).
\end{align*}
The sum over $w$ is zero unless $h \in W^\perp$. Note that $\dim W^\perp = s$. Hence
\[
  |\hat{\Phi}(\chi)|^2 \leq |W|\,|W^\perp| = q^n.
\]
By Fourier inversion we have
\[
  \Phi(x)
  = q^{n-s-1} + \frac1q \sum_{1\neq \chi \in \hat{\F_q}} \hat{\Phi}(\chi) \chi(x),
\]
so
\[
  |\Phi(x) - q^{n-s-1}| \leq \frac1q \sum_{1 \neq \chi \in \hat{\F_q}} |\hat{\Phi}(\chi)| \leq q^{n/2}.\qedhere
\]
\end{proof}

Relatedly, we have Witt's lemma, which characterizes the orbits of $\GCl_n(q)$ in terms of $f$ and $Q$.

\begin{lemma}[Witt's lemma] \label{lem:witt}
Let $u_1, \dots, u_k, v_1, \dots, v_k \in V$ be vectors such that
\begin{align}
  \dim \langle u_1, \dots, u_k \rangle &= \dim \langle v_1, \dots, v_k\rangle \\
  f(u_i, u_j) &= f(v_i, v_j) && (1 \leq i,j \leq k) \\
  Q(u_i) &= Q(v_i) && (1 \leq i \leq k).
\end{align}
Then there is an element $g \in \GCl_n(q)$ such that $g u_i = v_i$ for each $1 \leq i \leq k$. If $k \leq n-2$ there is such an element in $\SCl_n(q)$.
\end{lemma}
\begin{proof}
See, e.g., \cite[Section 20]{aschbacher}.
\end{proof}

\subsection{Degree and support}

The concepts of degree and support are essential in the rest of the paper.
Both concepts are analogous to the size of the support of a permutation, defined as the set of non-fixed points.
The \emph{degree} of an element $g \in \GL_n(q)$ is
\[
  \deg g = \rank (g - 1);
\]
the \emph{support} of $g \in \GL_n(q)$ is
\[
  \supp g = \min_{\lambda \in \barF} \rank (g - \lambda)
\]
(the former definition follows \cite{biswas--yang} and \cite{HMPQ}; the latter definition follows Larsen--Shalev--Tiep~\cite{larsen-shalev-tiep-waring}). Equivalently, if $V_\lambda = \ker (g - \lambda)$ denotes the $\lambda$-eigenspace of $g$ (for $\lambda \in \barF$), then
\begin{align}
  \deg g &= \codim V_1, \\
  \supp g &= \min_{\lambda \in \barF} \codim V_\lambda.
\end{align}

Support is closely related to the size of the centralizer, as in the following lemma.

\begin{lemma}
  \label{lem:support-and-centralizer}
  For $g \in G \leq \GL_n(q)$,
  \[
    |C_G(g)| \leq q^{n (n - \supp g)}.
  \]
\end{lemma}
\begin{proof}
(Cf.~\cite[Lemma~3.1]{larsen-shalev-fibers}.)
Clearly
\[
  |C_G(g)| \leq |C_{\M_n(\F_q)}(g)|.
\]
Note that $C_{\M_n(\F_q)}(g)$ is a vector space over $\F_q$, so it will suffice to bound its dimension.
Consider $g$ as an element of $\GL_n(\barF)$ and decompose it into Jordan blocks. For each eigenvalue $\lambda$ of $g$, let $\pi_\lambda$ be the partition whose parts are the sizes of Jordan blocks associated to $\lambda$.
Denote by $S^i(\pi)$ the sum of $i$th powers of the parts of a partition $\pi$
and let $\pi'$ be the transposed partition of $\pi$.
By \cite[Section 1.3]{humphreys1995conjugacy},
\[
  \dim C_{\M_n(\F_q)}(g)
  =
  \sum_{\lambda} S^2(\pi_\lambda').
\]
The largest part of $\pi_\lambda'$ is the dimension of $V_\lambda$, so
\[
  S^2(\pi_\lambda') \leq S^1(\pi_\lambda') \dim V_\lambda.
\]
Combined with $\sum_{\lambda} S^1(\pi_\lambda') = n$, this implies
\[
  \dim C_{\M_n(\F_q)}(g) \leq n \max_\lambda \dim V_\lambda = n (n - \supp g). \qedhere
\]
\end{proof}

\subsection{Adjacency operator}
\label{subsec:adjop}

Given any group $G$ and $x_1, \dots, x_k \in G$, let
\[
  \adjcy = \adjcy_{x_1, \dots, x_k} = \frac1{2k} \sum_{i=1}^{k} (x_i + x_i^{-1}).
\]
This is an element of the group algebra $\C[G]$.
Given any $\C[G]$-module $W$, we may consider the action of $\adjcy$ on $W$. Since $\adjcy$ is self-adjoint its spectrum is real.
Write $\rho(\adjcy, W)$ for the spectral radius of $\adjcy$.

We are most interested in permutation modules. If $G$ acts transitively on a set $\Omega$ then there is a corresponding permutation module $\C[\Omega]$ containing a single copy of the trivial representation, denoted $\C[\Omega]^G$.
Let $W = \C[\Omega]_0$ denote the orthogonal complement of $\C[\Omega]^G$.
The \emph{spectral gap} is $1 - \rho(\adjcy, W)$.
Equivalently, if $\adjcy$ acting on $\C[\Omega]$ has spectrum
\[
  1 = \lambda_1 \geq \lambda_2 \cdots \geq \lambda_N \geq -1,
\]
where $N = |\Omega|$, then
\[
  \rho(\adjcy, W) = \max(\lambda_2, -\lambda_N),
\]
so the spectral gap is
\[
  \min(1 - \lambda_2, 1 - |\lambda_N|).
\]
We say the action of $x_1, \dots, x_k$ on $\Omega$ is \emph{expanding} if the spectral gap is bounded away from zero. This is equivalent to rapid mixing of the random walk on $\Omega$.

\section{Word maps, queries, and trajectories}
\label{sec:word_maps}

\subsection{Word maps}

Write $F_k = F\alphabet$ for the free group with generators $\alphabet$.
Let $w \in F_k$ have length $\ell$, and let
\[
  w = w_\ell \cdots w_1 \qquad (w_i \in \signedalphabet)
\]
be the reduced expression of $w$.
Let $G$ be a finite group and $x_1, \dots, x_k \in G$.
Write
\[
  \bar w = w(x_1, \dots, x_k)
\]
for the image of $w$ under the homomorphism $F_k \to G$ defined by $\xi_i \mapsto x_i$.

Usually, but not always, $x_1, \dots, x_k$ will be chosen randomly. The following lemma is often useful for reducing to the cyclically reduced case.

\begin{lemma} \label{lem:barw-coordinate-free}
If $x_1, \dots, x_k \in G$ are uniform and independent then $\bar w$ is just the image of $w$ under a uniformly random homomorphism $F_k \to G$. In particular, the distribution of $\bar w$ depends only on the automorphism class of $w$.
\end{lemma}

\subsection{Queries and coincidences}

Let $G = \Cl_n(q)$ be a classical group and $V = \F_q^n$ the defining module.
Let $x_1, \dots, x_k \in G$. Define a \emph{query} to be a pair $(\xi, v)$, where $\xi \in\signedalphabet$ and $v \in V$;
the \emph{result} of the query is $\bar \xi v$.
After any finite sequence of queries
\[
  (w_1, v_1), (w_2, v_2), \dots, (w_{t-1}, v_{t-1})
\]
the \emph{known domain} of a letter $\xi$ at time $t$ is
\[
  D_\xi^t = \sp\{ v_i : w_i = \xi, i < t\} + \sp\{\bar{w_i} v_i : w_i = \xi^{-1}, i < t\}.
\]
Suppose we make a further query $(w_t, v_t)$.
If $v_t \in D_{w_t}^t$, then the result $\bar {w_t} v_t$ is determined already by the values of $\bar{w_1} v_1, \dots, \bar{w_{t-1}} v_{t-1}$; we call this a \emph{forced choice}.
Otherwise, we say the query is a \emph{free choice}.

Let $R$ be some subset of $V$ fixed in advance. If a query $(w_t, v_t)$ is a free choice and yet
\[
  \bar {w_t} v_t \in \sp R + \sp \{v_1, \bar {w_1} v_1, \dots, v_{t-1}, \bar {w_{t-1}} v_{t-1}, v_t\}
\]
then we say the result of the query is a \emph{coincidence}.

The language is most interesting when $x_1, \dots, x_k \in G$ are chosen randomly. Then, by Witt's lemma, whenever $(\xi, v)$ is a free choice, $\bar \xi v$ is, conditionally on the result of previous queries, uniformly distributed among vectors satisfying the relevant independence and form conditions. In particular, coincidences are unlikely. We formalize these key points in the following lemmas.

\begin{lemma}
Let $x \in G$ be uniformly random, and let $u_1, \dots, u_t$ be linearly independent, where $t \leq n-2$. Then, conditionally on the values of $v_1 = x u_1, \dots, v_{t-1} = x u_{t-1}$, the value of $x u_t$ is uniformly distributed among vectors $v_t$ such that $u_i \mapsto v_i$ defines an isometric isomorphism $\langle u_1, \dots, u_t\rangle \to \langle v_1, \dots, v_t \rangle$, or in other words such that $v_t \notin \sp\{v_1, \dots, v_{t-1}\}$ and $f(u_i, u_t) = f(v_i, v_t)$ for each $i\leq t$ and $Q(u_t) = Q(v_t)$.
\end{lemma}
\begin{proof}
For each such $v_t$, Witt's lemma asserts that there is at least one suitable $x \in G$. The distribution is uniform by the orbit--stabilizer theorem.
\end{proof}

\begin{lemma} \label{lem:free-choices}
Let $x_1, \dots, x_k \in G$ be uniformly random and independent, and let
\[
(w_1, v_1), (w_2, v_2), \dots, (w_{t-1}, v_{t-1})
\]
be a sequence of queries.
Assume that $(w_t, v_t)$ is a free choice.
Assume
\[
  \dim \langle v_1, \dots, v_t\rangle \leq n-2.
\]
Then, conditionally on the values of $\bar{w_1} v_1, \dots, \bar{w_{t-1}} v_{t-1}$,
the result $\bar{w_t} v_t$ of the query $(w_t, v_t)$ is uniformly distributed outside $D_{w_t^{-1}}^t$ subject to
\begin{align*}
  f(\bar{w_i} v_i, \bar{w_t} v_t)
  &= f(v_i, v_t) && (i < t, w_i = w_t),\\
  f(v_i, \bar{w_t} v_t)
  &= f(\bar{w_i} v_i, v_t) && (i < t, w_i = w_t^{-1}),\\
  Q(\bar{w_t} v_t) &= Q(v_t).
\end{align*}

In particular, the conditional probability that $\bar {w_t} v_t$ is a  coincidence is bounded by
\[
  \frac{q^d}{q^{n-s}/q_0 - q^s - q^{n/2}}
\]
(provided the denominator is positive),
where
\[
  d = \dim (\sp R + \sp\{v_1, \bar{w_1} v_1, \dots, v_{t-1}, \bar{w_{t-1}} v_{t-1}, v_t\})
\]
and $s$ is the number of $i < t$ with $w_i \in \{w_t, w_t^{-1}\}$.
\end{lemma}
\begin{proof}
The first part of the lemma is immediate from the previous lemma. For the second part, note that $\bar{w_t} v$ is drawn from an affine subspace of codimension at most $s$, less a subspace of dimension at most $s$, subject only to the quadratic condition; by Lemma~\ref{lem:subspace-counting-lemma} there are at least $q^{n-s}/q_0 - q^{n/2} - q^s$ possibilities, so we get at least the denominator claimed.
\end{proof}

\begin{remark}
In the linear case there are no form conditions, so we get the simpler bound $q^d / (q^n - q^s)$ for the probability of a coincidence.
\end{remark}

\subsection{Trajectories}
\label{subsec:trajectories}

Let $w \in F_k$, and let
\[
  w = w_\ell \cdots w_1 \qquad (w_i \in \signedalphabet)
\]
be the reduced expression.
For each $v \in V$, the \emph{trajectory} of $v$ is the sequence of queries $(w_t, v^{t-1})$, where $v^0 = v$ and for each $t \geq 1$ the vector $v^t$ is the result of the query $(w_t, v^{t-1})$; in other words, the sequence $v^0, v^1, \dots, v^\ell$ is defined by
\begin{align*}
  v^0 &= v, \\
  v^t &= \bar {w_t} v^{t-1} && (1\le t\le \ell).
\end{align*}
The following lemma is trivial but essential.

\begin{lemma} \label{lem:at-least-one-coincidence}
Suppose $v \neq 0$ and $v^\ell \in \sp R$. Then there is at least one coincidence in the trajectory of $v$.
\end{lemma}
\begin{proof}
Since $D_{w_1}^1 = 0$, the first query $(w_1, v^0)$ is free. For each $t\geq 1$, if $(w_t, v^{t-1})$ is free and not a coincidence then
\[
  v^t = \bar {w_t} v^{t-1} \notin \sp R + \sp \{v^0, \dots, v^{t-1}\},
\]
while
\[
  D_{w_{t+1}}^{t+1} \leq \sp \{v^0, \dots, v^{t-1}\};
\]
hence the query $(w_{t+1}, v^t)$ is also free.
Finally if $(w_\ell, v^{\ell-1})$ is free and not a coincidence then $v^\ell \notin \sp R$.
\end{proof}

More generally
for any $r\geq1$
we consider the \emph{joint trajectory} of an $r$-tuple
\[
  (v_1, \dots, v_r) \in V^r,
\]
which is simply the $r$-tuple of individual trajectories,
with the queries $(w_t, v_i^{t-1})$ ordered lexicographically by $(t, i)$; i.e., we answer the queries
\begin{align*}
  (w_1, v_1^0) && (w_1, v_2^0) && \cdots && (w_1, v_r^0) \\
  (w_2, v_1^1) && (w_2, v_2^1) && \cdots && (w_2, v_r^1) \\
  && \vdots &&&&
\end{align*}
in reading order.
Write $\prec$ for this order, i.e., $(t',i') \prec (t,i)$ if $t' < t$ or $t'=t$ and $i' < i$.
The following lemma generalizes the previous one.

\begin{lemma} \label{lem:at-least-one-coincidence-joint}
Suppose $v_i \notin \sp\{v_1, \dots, v_{i-1}\}$ and $v_i^\ell \in \sp R$. Then there is at least one coincidence in the trajectory of $v_i$ (during the joint trajectory of $v_1, \dots, v_r$).
\end{lemma}
\begin{proof}
At time $(1, i)$, we have
\[
  D_{w_1}^{(1, i)} \leq \sp \{v_1, \dots, v_{i-1}\},
\]
so the first query $(w_1, v_i^0)$ is free.
For each $t\geq 1$, if $(w_t, v_i^{t-1})$ is free and not a coincidence then
\[
  v_i^t
  = \bar {w_t} v_i^{t-1}
  \notin \sp R + \sp \{v_{i'}^{t'} : (t', i') \prec (t, i)\}
\]
(the vectors $v_{i'}^{t'}$ with $t' = t$ and $i' < i$ get included because they are results of previous queries), while
\[
  D_{w_{t+1}}^{(t+1, i)} \leq \sp \{v_{i'}^{t'} : (t', i') \prec (t, i)\};
\]
hence the query $(w_{t+1}, v_i^t)$ is also free.
Finally if $(w_\ell, v_i^{\ell-1})$ is free and not a coincidence then $v_i^\ell \notin \sp R$.
\end{proof}

\section{The probability of small support}
\label{sec:prob-small-support}

Let $G$ be a finite group, let $w \in F_k$, let $x_1, \dots, x_k \in G$ be random, and consider $\bar w = w(x_1, \dots, x_k)$. The probability that $\bar w = 1$ quantifies the extent to which $w$ is ``almost a law'' in $G$. This probability is a well-studied quantity, particularly when $G$ is simple. For example, it is known that for any $w \neq 1$ there is some $c = c(w) > 0$ such that $\P(\bar w = 1) \leq |G|^{-c}$ for all sufficiently large finite simple groups $G$ (Larsen--Shalev~\cite[Theorem~1.1]{larsen-shalev-fibers}).

For groups of large rank (our particular interest), the following bounds have been proved recently. Let $\ell > 0$ be the reduced length of $w$.
\begin{enumerate}
  \item For $G = A_n$ or $G = S_n$, if $\ell < cn^{1/2}$ then $$\P(\bar w=1) \leq e^{-c n / \ell^2}$$ (Eberhard~\cite[Lemma~2.2]{eberhard-Sn-girth}).
  \item For any classical group $G = \Cl_n(q)$, if $\ell < cn$ then $$\P(\bar w=1) \leq |G|^{-c/\ell}$$ (Liebeck--Shalev~\cite[Theorem~4]{liebeck-shalev-Cln-girth}).
\end{enumerate}

The proofs of these estimates can be adapted to show more, namely that with high probability $\bar w$ has large support. In this section we explain this observation in detail in the case of $G = \Cl_n(q)$. For the case of $G = A_n$ or $G = S_n$, see the appendix (Subsection~\ref{subsec:Sn-small-support}).

The following lemma generalizes a key step from the argument of \cite[Theorem~4]{liebeck-shalev-Cln-girth}.

\begin{lemma}
Let $G = \Cl_n(q)$ be a classical group of dimension $n$.
Let $V = \F_q^n$ be the natural module, and let $U \leq V$ be a subspace of dimension $r \leq n-2$.
Let $w \in F_k$ be a nontrivial word of length $\ell \leq (\frac{n}{2} - 2)/r$.
Then
\[
  \P\br{\bar w U = U} \leq \br{ C_{q^r} \frac{q^{\ell r}}{q^{n - \ell r - 1} - q^{\ell r} - q^{n/2}}}^r,
\]
where $C_{q^r} = 1 + (1 - q^{-r})^{-1} \leq 3$.
\end{lemma}
\begin{proof}
Let $v_1, \dots, v_r$ be a basis for $U$. Consider the joint trajectory of $v_1, \dots, v_r$. By Lemma~\ref{lem:at-least-one-coincidence-joint} with $R = \{v_1, \dots, v_r\}$, we can have $\bar w U = U$ only if there is at least one coincidence in each individual trajectory. We take a union bound over all possibilities for when the coincidences could occur. If $t < \ell$, then by Lemma~\ref{lem:free-choices}, the probability that step $(t, i)$ is a coincidence is bounded by
\[
  \frac{q^{(t+1) r}}{q^{n-\ell r-1} - q^{\ell r} - q^{n/2}};
\]
indeed there are at most $t r + i \leq (t+1) r \leq \ell r$ previous vectors. If $t = \ell$, assuming $v_j^\ell \in U$ for $j < i$, we actually get a slightly stronger bound:
\[
  \frac{q^{\ell r}}{q^{n - \ell r - 1} - q^{\ell r} - q^{n/2}}.
\]
Summing over $t$, the probability that there is a coincidence in the trajectory of $v_i$ is bounded by
\[
  (1 + 1 + q^{-r} + q^{-2r} + \cdots) \frac{q^{\ell r}}{q^{n - \ell r - 1} - q^{\ell r} - q^{n/2}}.
\]
Taking the product over $i$ gives the claimed bound.
\end{proof}

In the following proof we will refer to the ``$q$-binomial coefficient'', defined by
\[
  \binom{x}{r}_q = \frac{(q^x - 1) (q^x - q) \cdots (q^x - q^{r-1})}{(q^r - 1) (q^r - q) \cdots (q^r - q^{r-1})}.
\]
When $x$ is a nonnegative integer this is the number of $r$-dimensional subspaces of $\F_q^x$. For $x \geq r$ note that $x\mapsto \binom{x}{r}_q$ is increasing and nonnegative, and
\begin{equation} \label{eq:q-binom-estimate}
  \binom{x}{r}_q
  = q^{xr - r^2} \frac{(1-q^{-x+r-1}) \cdots (1 - q^{-x})}{(1-q^{-r}) \cdots (1-q^{-1})}
  \asymp q^{xr - r^2}.
\end{equation}

The following theorem will be used for an unspecified, but fixed, $\delta > 0$.

\begin{theorem} \label{theorem:supp0.99}
There are constants $c, C>0$ such that the following holds for all $\delta > 0$. Let $G = \Cl_n(q)$ be a classical group of dimension $n$, and let $w \in F_k$ be a nontrivial word of reduced length $\ell < \delta^2 n / 20$.
Assume $q^{\delta n} > C$.
Then
\[
  \P\br{\supp \bar w \leq (1-\delta)n} \leq |G|^{-c \delta^2/\ell}.
\]
\end{theorem}
\begin{proof}
Let $x_1, \dots, x_k$ be chosen independently and uniformly from $G$.
Suppose some eigenspace $V_\lambda \leq \barF^n$ of $\bar w$ has dimension at least $\delta n$.
Let $d = [\F_q(\lambda):\F_q]$.
Let $\Lambda$ be the set of $d$ Galois conjugates of $\lambda$.
Since $\dim V_{\lambda'} = \dim V_\lambda$ for each $\lambda' \in \Lambda$, $\dim V_\lambda \leq n / d$, so $d \leq 1 / \delta$.
Let $W \leq V_\lambda$ be an $r$-dimensional subspace defined over $\F_q(\lambda) \cong \F_{q^d}$.
Then there is a conjugate subspace $W' \leq V_{\lambda'}$ for each $\lambda' \in \Lambda$,
and the sum $U = \sum_{\lambda' \in \Lambda} W'$ is a $dr$-dimensional and $\F_q$-rational since it is fixed by the Galois group,
so it may be identified with a $dr$-dimensional subspace of $V$.
Since $U \cap V_\lambda = W$, this correspondence $W \mapsto U$ is injective.
Hence the number of $dr$-dimensional subspaces of $V$ preserved by $\bar w$ is at least $\binom{\delta n}{r}_{q^d}$.

Since $\ell d \leq \ell / \delta < \delta n / 20$, we may choose an integer $r > 0$ such that
$\ell d r \in [\delta n / 5, \delta n / 4]$.
Now by the previous lemma and Markov's inequality,
the probability that the number of $dr$-dimensional subspaces of $V$ preserved by $\bar w$ is at least $\binom{\delta n}{r}_{q^d}$
is bounded by
\begin{align*}
    \frac{\binom{n}{d r}_q}{\binom{\delta n}{r}_{q^d}} \br{
    3 \frac{q^{\ell d r}}{q^{n - \ell d r - 1} - q^{\ell d r} - q^{n/2}}
    }^{dr}
    &\asymp \frac{q^{drn - d^2 r^2}}{(q^d)^{\delta r n - r^2}}\br{
    3 \frac{q^{\ell d r}}{q^{n - \ell d r - 1} - q^{\ell d r} - q^{n/2}}
    }^{dr} \\
    &\leq 
    O\br{q^{-\delta n + 2 \ell dr + r - dr + 1}}^{dr}
    \\
    &\leq
    O(1)^{\delta n / 4 \ell} \br{q^{-\delta n + 2\delta n / 4 + \delta n / 4}}^{\delta n/5\ell}
    \\
    &= O(1)^{\delta n / \ell} q^{-\frac1{20} \delta^2 n^2 / \ell}.
\end{align*}
Taking the sum over all $d \leq 1/ \delta$, it follows that
\[
  \P\br{
  \supp\bar w \leq (1- \delta)n
  }
  = \P\br{
    \max_{\lambda \in \barF} \dim V_\lambda \geq \delta n
  }
  \leq \delta^{-1} O(1)^{\delta n / \ell} q^{-\frac1{20} \delta^2 n^2 / \ell}.
\]
Assuming $q^{\delta n}$ is sufficiently large, the first two factors are negligible compared to the third.
\end{proof}

\begin{remark}
The restriction $\ell < c \delta^2 n$ in Theorem~\ref{theorem:supp0.99} is essential, and related to our reliance on linear algebra. For example, let $G = \SL_n(q)$, and suppose $w$ is a word of length $\ell \approx 10 n$. We do not know how to bound $\P(\bar w = 1)$ satisfactorily. Is it true that $\P(\bar w = 1) \leq q^{-cn}$ for some $c>0$?
Certainly $w$ cannot be a law, because $\SL_n(q)$ contains $\SL_2(q^{\floor{n/2}})$ and the shortest law in $\SL_2(q^{\floor{n/2}})$ has length at least $(q^{\floor{n/2}} - 1)/3$ (see Hadad~\cite[Theorem~2]{hadad}).
The question is whether it can be an almost-law.
\end{remark}

\section{Expected values of characters}
\label{sec:expected_char_values_bound}

Throughout this section let $G = \Cl_n(q)$ be a classical group and $\chi \in \Irr G$ a nonlinear character.
Our aim is to bound
\[
  \E_{x_1, \dots, x_k} \br{\frac{|\chi(\bar w)|}{\chi(1)}}
\]
when $w$ is a fixed nontrivial word of length $cn$, evaluated at random $x_1, \dots, x_k \in G$. The proof consists of two steps:
\begin{enumerate}
  \item By the previous section, with high probability $\bar w$ has large support.
  \item By recent character bounds of Guralnick, Larsen, and Tiep~\cite{GLT1, GLT2}, if $\bar w$ has large support then $|\chi(\bar w)| \leq \chi(1)^\eps$.
\end{enumerate}

We first deal with elements of large support.

\begin{lemma} \label{lemma:large_support_character_bound}
For every $\eps>0$ there is a $\delta > 0$ such that the following holds.
Let $g \in G$ with $\supp g \geq (1-\delta)n$.
Then $|\chi(g)| \leq \chi(1)^\eps$.
\end{lemma}

\begin{proof}
By Lemma~\ref{lem:support-and-centralizer}, $|C_G(g)| \leq q^{\delta n^2}$.
Hence by the character bound \cite[Theorem~1.3]{GLT2} we have $|\chi(g)| \leq \chi(1)^\eps$.
\end{proof}

\begin{theorem} \label{theorem:e_yz_character_bound}
There is a constant $c > 0$ such that the following holds.
Let $w \in F_k$ be a fixed nontrivial word of reduced length less than $cn$.
Then
\[
  \E_{x_1, \dots, x_k}
  \left( \frac{|\chi(\bar w)|}{\chi(1)} \right)
  < q^{-c n}.
\]
\end{theorem}

\begin{proof}
Let $\delta$ be as in the previous lemma with $\eps = 1/2$.
By conditioning on whether or not $\supp{\bar w} < (1-\delta)n$, we have
\begin{align*}
\E_{x_1, \dots, x_k}
\br{\frac{|\chi(\bar w)|}{\chi(1)}}
&\leq
\P_{x_1, \dots, x_k}\br{\supp{\bar w} < (1-\delta)n} \\
&\qquad +
\max_{x_1, \dots, x_k : \supp{\bar w} \geq (1-\delta)n}
\left( \frac{|\chi(\bar w)|}{\chi(1)} \right).
\end{align*}
It follows from Theorem~\ref{theorem:supp0.99} that
\[
\P_{x_1, \dots, x_k}(\supp{\bar w} < (1-\delta)n) \leq q^{-c_1n}
\]
for some constant $c_1 > 0$.
The other summand is bounded by Lemma~\ref{lemma:large_support_character_bound}:
\[
\max_{x_1, \dots, x_k \colon \supp{\bar w} \geq (1-\delta)n}
\left( \frac{|\chi(\bar w)|}{\chi(1)} \right)
\leq
\chi(1)^{-1/2}
\leq
q^{-c_2 n}
\]
for some constant $c_2 > 0$. (Here we used $\chi(1) \geq q^{c_3 n}$: see \cite{landazuri1974minimal}.)
\end{proof}

Our main interest is the case in which $w$ is the result of a simple random walk in $F_k$. With high probability the result of the random walk is nontrivial, so we can apply the above theorem.

\begin{corollary} \label{cor:expected_character_bound_exponential}
There is a constant $c > 0$ such that the following holds.
Let $w$ be the result of a simple random walk of length $\ell < cn$ in $F_k$.
Then
\[
  \E_{x_1, \dots, x_k\in G, w}
  \left( \frac{|\chi(\bar w)|}{\chi(1)} \right)
  < q^{-c n} + k^{-c \ell}.
\]
\end{corollary}

\begin{proof}
By conditioning on whether or not the word $w$ is trivial, we get
\[
\E_{x_1, \dots, x_k, w}
\left( \frac{|\chi(\bar w)|}{\chi(1)} \right)
\leq
\max_{0 < |w| < cn} \E_{x_1, \dots, x_k}
\left( \frac{|\chi(\bar w)|}{\chi(1)} \right)
+
\P_{w}(w = 1).
\]
The first term is bounded by Theorem~\ref{theorem:e_yz_character_bound}.
The second term is the return probability of a simple random walk on a $2k$-regular tree,
which is at most $k^{-c \ell}$ for a constant $c > 0$ (see \cite[Theorem 3 and Lemma 2.2]{kesten1959symmetric} or \cite[Appendix~B]{FJRST}).
\end{proof}

\section{Reaching a normal subset: the \texorpdfstring{$x w(y, z)$}{x w(y, z)} trick}
\label{sec:reaching_normal_subset}

In this section, something of an interlude, let $G$ be any finite group, and let $\CC$ be a normal (i.e., conjugacy-closed) subset of a group $G$. We will develop a
criterion ensuring that one can,
with high probability as $x,y,z \in G$ are chosen uniformly at random,
find a word $w \in F_2$ of at most a prescribed length such that $x w(y, z) \in \CC$.
The criterion applies to sets $\CC$ whose density is large compared to the expected values of characters.
This is a variation of the technique used in
\cite[Section~4]{EV-Sn}; see also \cite[Section~2]{EV-SL}.

The following theorem expresses the most general such estimate we will need, in which we further allow arbitrary weights to be attached to elements of $\CC$. We express the result in terms of a nonnegative conjugation-invariant function (class function) $f$ on $G$. We define the $L^p$ norm of $f$ by
\[
  \|f\|_p^p = \frac1{|G|} \sum_{x \in G} |f(x)|^p \qquad (p \in \{1, 2\}),
\]
and we use the standard inner product on functions on $G$ defined by
\[
    \langle f, g \rangle = \frac{1}{|G|} \sum_{x \in G} f(x) \bar{g(x)}.
\]

\begin{theorem}
\label{thm:general-xw(y,z)-trick}
Let $f$ be a nonnegative and conjugation-invariant function on $G$,
and let $\ell$ be a positive integer.
Let $x_0, x_1, \dots, x_k$ be elements of $G$ chosen uniformly at random.
Let $E$ be the event that $f(x_0 \bar u) = 0$ for every word $u \in F_k$ of length at most $\ell$.
Let $w$ be the result of a simple random walk of length $2\ell$ in $F_k$.
Then\footnote{Note that the distribution of $\bar w$ is symmetric, so $\E_{x_1, \dots, x_k, w} \chi(\bar w) / \chi(1)$ is real.}
\[
  \P_{x_0, \dots, x_k} \br{E}
  \leq \frac{1}{\|f\|_1^2} \sum_{1 \neq \chi \in \Irr{G}} |\langle f, \chi\rangle|^2 \E_{x_1, \dots, x_k, w} \br{\frac{\chi(\bar w)}{\chi(1)}}.
\]
In particular,
\[
  \P_{x_0, \dots, x_k} \br{E}
  \leq \frac{\|f\|_2^2}{\|f\|_1^2}
  \max_{\substack{1 \neq \chi \in \Irr{G} \\ \langle f, \chi \rangle \neq 0}}
  \E_{x_1, \dots, x_k,w} \br{\frac{\chi(\bar w)}{\chi(1)}}.
\]
\end{theorem}

\begin{proof}
Let $\adjcy = \adjcy_{x_1, \dots, x_k}$ be the adjacency operator defined in Subsection~\ref{subsec:adjop}, and consider its natural action on $L^2(G)$.
Let $X = \adjcy^\ell f(x_0)$, regarded as a random variable dependent on $x_0, x_1, \dots, x_k$, and note that $E$ is precisely the event $X = 0$.
By Chebyshev's inequality,
\begin{equation}
  \label{eq:chebyshev}
  \P(X = 0) \leq
  \frac{\Var X}{\left( \E X \right)^2}.
\end{equation}
The first moment is
\[
  \E X = \|f\|_1.
\]
The second moment is
\begin{equation}
  \E X^2
  = \E_{x_1, \dots, x_k} \|\adjcy^\ell f\|_2^2
  = \E_{x_1, \dots, x_k} \langle \adjcy^{2\ell} f, f\rangle \label{second-moment}
  .
\end{equation}
Since $f$ is conjugation-invariant, we can expand this further in terms of characters.
By orthogonality of characters, if $\tau_x$ is the translation operator defined by $\tau_x(h)(y) = h(x^{-1} y)$, we have
\[
  \langle \tau_x \chi, \psi\rangle =
  \begin{cases}
    \chi(x^{-1})/\chi(1) &\text{if}~\chi=\psi, \\
    0 &\text{else}.
  \end{cases}
\]
Hence
\[
  \langle \adjcy^{2\ell} \chi, \psi\rangle = 0 \qquad (\chi \neq \psi),
\]
and
\[
  \langle \adjcy^{2\ell} \chi, \chi \rangle = \E_w \br{\frac{\chi(\bar w)}{\chi(1)}},
\]
where $w$ is the result of a simple (symmetric) random walk of length $2\ell$ in $F_k$.
Hence, from \eqref{second-moment},
\[
  \E X^2 = \sum_{\chi \in \Irr{G}} |\langle f, \chi \rangle|^2 \E_{x_1, \dots, x_k, w} \br{\frac{\chi(\bar w)}{\chi(1)}}.
\]
The $\chi = 1$ term is $\|f\|_1^2$, which is the same as $(\E X)^2$.
Hence the first part of the theorem follows from \eqref{eq:chebyshev}.
The second part holds because
\[
  \sum_{\chi \in \Irr{G}} |\langle f, \chi \rangle|^2 = \|f\|_2^2.\qedhere
\]
\end{proof}

\begin{corollary} \label{cor:xw(y,z)-trick}
Let $\CC$ be a normal subset of $G$.
Write
\[
  \CC = \bigcup_{\alpha \in G^\ab} \CC_\alpha,
\]
where $\CC_\alpha = \CC \cap \alpha G'$ is the fibre of $\CC$ over $\alpha \in G^\ab$.
Let $\delta_\alpha = |\CC_\alpha| / |G'|$ be the fibre density, and let $\delta = \min_{\alpha \in G^\ab} \delta_\alpha$. Assume $\delta > 0$.

Let $x_0, x_1, \dots, x_k \in G$ be chosen uniformly at random, and let $E$ be the event that for every word $u \in F_k$ of length at most $\ell$ we have $x_0 \bar u \notin \CC$.
Let $w$ be the result of a simple random walk of length $2\ell$ in $F_k$.
Then
\[
  \P(E)
  \leq
  \delta^{-1} \max_{\substack{\chi \in \Irr{G} \\ \chi(1) > 1}} \E_{x_1, \dots, x_k, w} \br{\frac{\chi(\bar w)}{\chi(1)}}.
\]
\end{corollary}

\begin{proof}
In the previous theorem, take
\[
  f = \sum_{\alpha \in G^\ab} \frac{1_{\CC_\alpha}}{\delta_\alpha}.
\]
Then $\|f\|_1 = 1$, and
\[
  \|f\|_2^2 = \frac1{|G^\ab|} \sum_{\alpha \in G^\ab} \delta_\alpha^{-1} \leq \delta^{-1}.
\]
Thus
\[
  \P(E) \leq \delta^{-1} \max_{\substack{1 \neq \chi \in \Irr{G} \\ \langle f, \chi \rangle \neq 0}} \E_{x_1, \dots, x_k, w} \br{\frac{\chi(\bar w)}{\chi(1)}}.
\]
Now if $\chi \neq 1$ is one-dimensional then $\chi$ factors through $G^\ab$, so
\[
  \langle f, \chi \rangle = \frac1{|G^\ab|} \sum_{\alpha \in G^\ab} \chi(\alpha) = 0.
\]
Hence
\[
  \P(E) \leq \delta^{-1} \max_{\substack{1 \neq \chi \in \Irr{G} \\ \chi(1) > 1}} \E_{x_1, \dots, x_k, w} \br{\frac{\chi(\bar w)}{\chi(1)}}.\qedhere
\]
\end{proof}

\section{Obtaining an element of minimal degree}
\label{sec:obtaining_minimal_support}

Let $G = \GCl_n(q)$.
Let $s$ be the minimal degree of a nontrivial element of $\SCl_n(q)$;
thus $s = 2$ in the orthogonal case and $s = 1$ otherwise.
Let 
\[
\MM = \{ g \in \SCl_n(q) : \deg g = s \}.
\]
In this section we exhibit a large normal subset $\CC_d \subseteq G$ with an integer parameter $d$ whose $q^d - 1$ power is contained in $\MM$.
We will use $\CC_d$ in combination with Corollaries~\ref{cor:expected_character_bound_exponential} and \ref{cor:xw(y,z)-trick}
to obtain an element of minimal degree as a short word
in random generators.

\begin{proposition} \label{prop:density_of_cc_d}
There is a constant $C > 0$ so that the following holds.
Let $d \in [2, n]$ be an integer parameter. Assume $q^d > Cn$.
Then there is a normal subset $\CC_d \subseteq G$
with the following properties.
\begin{enumerate}[(1)]
  \item For every $\alpha \in G^\ab$, if $\CC_{d;\alpha}$ is the fibre of $\CC_d$ over $\alpha$, then
  \[
    \frac{|\CC_{d; \alpha}|}{|G|} \geq \exp\br{ -O(d^2 \log q) - O(d^{-1} n \log n) }.
  \]
  \item For every $g \in \CC_d$, we have
  \[
  g^{\kappa (q^d - 1)} \in \MM,
  \]
  where $\kappa = 2$ if $G$ is orthogonal in even characteristic, and $\kappa = 1$ otherwise.
\end{enumerate}


\end{proposition}

The proof is split into cases depending on the type of $G$.

\subsection{The linear case}

Let $G = \GL_n(q)$. In this case $\MM$ is the set of transvections.
Let $V$ be the natural module for $G$.
Write
\[
  n - 3 = kd + r, \qquad (0 \leq r < d),
\]
i.e., let $k = \floor{\frac{n-3}{d}}$ and $r = n - 3 - k d$.
Decompose $V$ as
\[
V = L \oplus V_1 \oplus \cdots \oplus V_k \oplus R \oplus W,
\]
where $\dim L = 2$, $\dim V_i = d$, $\dim R = 1$, and $\dim W = r$.
Fix a basis for each of the subspaces.

We now define a particular element $g \in \GL(V)$
respecting the above decomposition.
We define $g$ by its action on the chosen basis for each of the subspaces above.

\begin{description}

\item[Subspace $L$] Let $g$ act as a transvection on $L$, say
$\begin{psmallmatrix}
1 & 1 \\ 0 & 1
\end{psmallmatrix}$.
Note that $(g|_L)^{q^d-1} = (g|_L)^{-1}$ is also a transvection.

\item[Subspace $V_i$] Let $p_i$ be a monic irreducible polynomial of degree $d$ over $\F_q$. Identify $V_i$ with $\F_q[t]/(p_i(t))$. The variable $t$ acts on the latter space by multiplication. Let $g$ act on $V_i$ as multiplication by $t$. Note that the minimal polynomial of this transformation is $p_i$, and $(g|_{V_i})^{q^d-1} = 1$.

\item[Subspace $R$]
Let $\alpha \in G^\ab$.
Let $g$ act on $R$ as the scalar $\det(\alpha)/ \prod_{i = 1}^k (-1)^d p_i(0)$.

\item[Subspace $W$] Let $g$ act trivially on $W$.

\end{description}

Let $\irrpolys_d$ denote the set of monic irreducible polynomials of degree $d$ over $\F_q$.
For every tuple $p_1, \dots, p_k \in \irrpolys_d$ with $p_{i} \neq p_{i'}$ for $i \neq i'$ and $\alpha \in G^\ab$ we thus have an element $g = g_{p_1, \dots, p_k; \alpha} \in G$.
Let $g_{p_1, \dots, p_k; \alpha}^G$ denote the conjugacy class of $g_{p_1, \dots, p_k; \alpha}$ (this class does not depend on the order of $p_1, \dots, p_k$).
Let
\[
  \CC_{d; \alpha} =
    \bigcup_{
    \{ p_1, \dots, p_k \} \in \binom{\irrpolys_d}{k}
    }
    g_{p_1, \dots, p_k; \alpha}^G.
\]
The union is disjoint, because the minimal polynomial of each element of $g_{p_1, \dots, p_k; \alpha}^G$ is divisible by $p_1(t) \cdots p_k(t)$ (the other factors are $(t-1)^2$ and $(t-\lambda)$ for $\lambda = \det(\alpha) / \prod_{i=1}^k (-1)^d p_i(0)$ if $\lambda \neq 1$).
Finally let
\[
  \CC_d = \bigcup_{\alpha \in G^\ab} \CC_{d; \alpha}.
\]
\begin{remark}
This is a variation of the construction in \cite[Section~3.2]{EV-SL}.
\end{remark}

\begin{proof}[Proof of Proposition \ref{prop:density_of_cc_d} for $\GL$]
By construction, $\CC_{d;\alpha}$ is the fibre of $\CC_d$ over $\alpha$, and for every $p_1, \dots, p_k,\alpha$ we have $g_{p_1, \dots, p_k; \alpha}^{q^d-1} \in \MM$. It remains only to estimate the density of $\CC_{d;\alpha}$.

For $g = g_{p_1, \dots, p_k; \alpha}$, we have (as in the proof of Lemma~\ref{lem:support-and-centralizer})
\[
  |C_G(g)| \leq
  |C_{\M_n(\F_q)}(g)| =
  q^{n + r^2 + O(r)}.
\]
Therefore
\begin{equation} \label{eq:size_of_cd_lower_bound_sl}
  |\CC_{d; \alpha}| \geq
  \binom{|\irrpolys_d|}{k} \cdot \frac{|G|}{q^{n + r^2 + O(r)}}.
\end{equation}
Recall that
\[
  |\irrpolys_d| = q^d / d - O(q^{d/2} / d).
\]
In particular, by the hypothesis $q^d > Cn$ we have $|\irrpolys_d| > k$, and in fact
\begin{align*}
  \binom{|\irrpolys_d|}{k}
  &=
  \binom{q^d/d - O(q^{d/2}/d)}{k} \\
  &\geq
  \left( \frac{ q^d/d - O(q^{d/2}/d) }{k} \right)^k \\
  &=
  \left( \frac{ q^d (1 - O(q^{-d/2}) }{dk} \right)^k \\
  &\geq
  \frac{q^{n-2-r}}{ n^{n/d} }
  \cdot e^{- O \left( n/ d \cdot q^{-d/2} \right) }.
\end{align*}
Hence, from \eqref{eq:size_of_cd_lower_bound_sl}, since $r < d$,
\[
  \frac{|\CC_{d; \alpha}|}{|G|}
  \geq
  \exp \left( - (d^2 + O(d)) \log q - \frac{n}{d} \log n - O(n/d \cdot q^{-d/2}) \right).
\]
This proves the proposition.
\end{proof}

\subsection{Other classical groups}

\def\an{\textup{an}}

Let $G = \GCl_n(q)$, where $\GCl \neq \GL$.
Let $V$ be the natural module for $G$ equipped with a nondegenerate binary form $f$ and possibly a quadratic form $Q$.
By Witt's decomposition theorem,
there is an orthogonal decomposition of $V$ of the form
\begin{equation} \label{eq:v_splitting}
V = H \perp V_\an,
\end{equation}
where $H$ is an orthogonal direct sum of hyperbolic planes
and $V_{\an}$ is anisotropic,
and $\dim V_\an \leq 2$ by the Chevalley--Warning theorem.
Let $\delta = \dim V_\an + 4 + 2 \kappa$,
where $\kappa = 2$ if $G$ is orthogonal in even characteristic, and $\kappa = 1$ otherwise.
Let $D = 2d$ and write
\[
  n - \delta = kD + r, \qquad (0 \leq r < D),
\]
i.e., let $k = \floor{(n-\delta)/D}$ and $r = n - \delta - k D$.
Write the hyperbolic space $H$ as
\[
  H = L \perp V_1 \perp \cdots \perp V_k \perp R \perp W',
\]
where each constituent is an orthogonal direct sum of hyperbolic planes with $\dim L = 2 \kappa + 2$, $\dim V_i = D$, $\dim R = 2$, and $\dim W' = r$. 
Let $W = W' \perp V_{\an}$. Thus we have the following
orthogonal decomposition of $V$:
\begin{equation}
  \label{eq:v_splitting_companion}
  V = L \perp V_1 \perp \cdots \perp V_k \perp R \perp W.
\end{equation}
Fix a hyperbolic basis for each of the hyperbolic spaces,
and fix a basis for $W$.

We now define a particular element $g \in \GCl(V)$
respecting the decomposition \eqref{eq:v_splitting_companion}.
As before we will define $g$ by its action on the chosen bases.

\begin{description}
  \item [Subspace $L$]
Let $v_1, \dots, v_{\kappa + 1}, w_1, \dots, w_{\kappa + 1}$ be the chosen hyperbolic basis for $L$,
i.e., such that $L_1 = \langle v_1, \dots, v_{\kappa + 1} \rangle$ and $L_2 = \langle w_1, \dots, w_{\kappa + 1} \rangle$ are totally singular subplanes, and $f$ is represented with respect to $v_1, \dots, v_{\kappa + 1}, w_1, \dots, w_{\kappa + 1}$ by
\[
\begin{pmatrix}
  0 & I \\
  \pm I & 0
\end{pmatrix}.
\]

\begin{description}
  \item [Symplectic case] Let $g$ act on $L$ as the transvection
  \[
    \begin{pmatrix}
    1 & 0 & 1 & 0 \\
    0 & 1 & 0 & 0 \\
    0 & 0 & 1 & 0 \\
    0 & 0 & 0 & 1 \\
    \end{pmatrix}.
  \]
  \item [Unitary case] Pick $\lambda \in \F_q$ be such that $\lambda + \lambda^\theta = 0$ (where $\theta$ is the field automorphism) and let $g$ act on $L$ as the transvection
  \[
    \begin{pmatrix}
    1 & 0 & \lambda & 0 \\
    0 & 1 & 0 & 0 \\
    0 & 0 & 1 & 0 \\
    0 & 0 & 0 & 1 \\
    \end{pmatrix}.
  \]
  \item [Orthogonal case] Let $g|_L$ be represented by the matrix
  \[
    \begin{pmatrix}
      A & 0 \\
      0 & A^{-T}
    \end{pmatrix},
  \]
  where in odd characteristic
  \[
    A =
    \begin{pmatrix}
    1 & 1 \\
    0 & 1
    \end{pmatrix}
  \]
  and in even characteristic
  \[
    A =
    \begin{pmatrix}
    1 & 1 & 0 \\
    0 & 1 & 1 \\
    0 & 0 & 1
    \end{pmatrix}.
  \]
\end{description}
In all cases we have $g|_L \in \SCl(L)$ and $(g|_L)^{\kappa(q^d-1)} \neq 1$.

  \item [Subspace $V_i$]
Fix a monic irreducible polynomial $p_i \in \F_q[t]$ of degree $d$.
Let $v_1, \dots, v_d, w_1, \dots, w_d$ be the chosen hyperbolic basis for $V_i$.
Thus there is a decomposition
\begin{equation} \label{eq:vi_decomposition_isotropic}
V_i = V_{i,1} \oplus V_{i,2}
\end{equation}
into totally singular subspaces
$V_{i,1} = \langle v_1, \dots, v_d \rangle$
and
$V_{i,2} = \langle w_1, \dots, w_d \rangle$
with $f(v_a, w_b) = \delta_{ab}$.
Identify $V_{i,1}$ with $\F_q[t]/(p_i(t))$.
The variable $t$ acts on the latter space by multiplication.
By Witt's lemma, this action extends to the space $V_i$.
This extension is moreover unique provided we demand that
it preserves the decomposition of $V_i$ (see \cite[Hilfssatz 3.1]{huppert1980isometrien}).
Let $g|_{V_i}$ be defined by this unique extension.

The minimal polynomial of this transformation can be determined as follows (see \cite{wall1963conjugacy}).
In the symplectic and orthogonal cases,
let $p^*(t) = p(0)^{-1} t^d p(t^{-1})$.
In the unitary case,
let $p^*(t) = p^\theta(0)^{-1} t^d p^\theta(t^{-1})$, where $\theta$ acts on the coefficients.
The minimal polynomial of $g$ acting on $V_i$ is $*$-symmetric,
divisible by $p_i$ (since $p_i$ is irreducible), and hence also divisible
by $p_i^*$.
Under the assumption that $p_i \neq p_i^*$,
the minimal polynomial of $g|_{V_i}$ must therefore be equal to $p_i p_i^*$. If $p_i = p_i^*$ then the minimal polynomial is $p_i$.

  \item [Subspace $R$]
Let $\alpha \in G^\ab$.
\begin{description}
\item[Symplectic case] Let $g$ act trivially on $R$. (Note $G^\ab$ is trivial.)
\item[Unitary case] Let $g$ act as the matrix
\[
  \begin{pmatrix}
    a & 0 \\ 0 & a^{-\theta}
  \end{pmatrix},
\]
where $a \in \F_q$ satisfies $a^{1-\theta} \prod_{i=1}^k p_i(0)^{1-\theta} = \det \alpha$.
Such an element always exists since $\det \alpha$ has norm $1$.
\item[Orthogonal case] The natural map $\GO(R)^\ab \to G^\ab$ is bijective.\footnote{Note that $\GO(R) \cong \GO_2^+(q) \cong D_{2(q-1)}$. In odd characteristic, $G^\ab \cong C_2 \times C_2$, and determinant and spinor norm are independent characters on $\GO_2^+(q)$. In even characteristic, $G^\ab \cong C_2$, and the Dickson invariant is nontrivial on $\GO_2^+(q)$.}
Let $g$ act on $R$ so that for every linear character $\lambda$ of $G$ we have
\[
  \lambda(g|_R) \prod_{i=1}^k \lambda(g|_{V_i}) = \lambda(\alpha).
\]
\end{description}
In all cases note that $(g|_R)^{\kappa(q^d-1)}$ is trivial.\footnote{The existence of an even-order linear character of $\GO_n(q)$ in even characteristic is why we need the extra factor of $2$ in that case.}
  \item [Subspace $W$] Let $g$ act trivially on $W$.
\end{description}

For every $k$-tuple $p_1, \dots, p_k \in \irrpolys_d$ and every $\alpha \in G^\ab$, we thus have an element $g = g_{p_1, \dots, p_k ; \alpha} \in G$.
The conjugacy class $g_{p_1, \dots, p_k ; \alpha}^G$ is invariant under reordering $p_1, \dots, p_k$, and under replacing any $p_i$ by $p_i^*$.
Conversely, $g_{p_1, \dots, p_k ; \alpha}^G$ is determined by $p_1(t) p_1^*(t) \cdots p_k(t) p_k^*(t)$ and $\alpha$.
Let $\irrpolys_d'$ be the set of unordered pairs $\{p, p^*\}$ of monic irreducible polynomials $p, p^* \in \irrpolys_d$ with $p \neq p^*$. 
Let
\[
  \CC_{d ; \alpha} = \bigcup
  \left\{
    g_{p_1, \dots, p_k ; \alpha}^G :
    \{ \{p_1, p_1^*\}, \dots, \{p_k, p_k^*\} \} \in \binom{\irrpolys_d'}{k}
  \right\},
\]
The union is disjoint, because the minimal polynomial of every element of $g_{p_1, \dots, p_k ; j}^G$ is divisible by $p_1(t) p_1^*(t) \cdots p_k(t) p_k^*(t)$ and has no other nonlinear factors. Finally let
\[
  \CC_d = \bigcup_{\alpha \in G^\ab} \CC_{d ; \alpha}.
\]

\begin{proof}[Proof of Proposition \ref{prop:density_of_cc_d} for other classical groups]
By construction, $g_{p_1,\dots,p_k ; \alpha}$ lies over $\alpha$ and $g_{p_1, \dots, p_k ; \alpha}^{\kappa(q^d-1)} \in \MM$. We must estimate the density of $\CC_{d;\alpha}$.

Consider $g = g_{p_1, \dots, p_k ; \alpha}$ for some $p_1, \dots, p_k \in \irrpolys_d'$ with $p_i \neq p_{i'}, p_{i'}^*$ for $i \neq i'$.
Let $h \in C_G(g)$. Then $h$ preserves each $V_{i,1}$ and $V_{i,2}$, those being the $p_i$- and $p_i^*$-primary subspaces of $g$. The restrictions of $h$ to $V_{i,1}$ and $V_{i,2}$ determine one another, and there are at most $q^d$ possibilities for $h|_{V_{i,1}}$ (as in Lemma~\ref{lem:support-and-centralizer}). Hence, since $\delta = O(1)$,
\[
  |C_G(g)| \leq (q^d)^k |\M_{r+\delta}(\F_q)|
  \leq q^{dk + r^2 + O(r)+O(1)}.
\]
Therefore
\begin{equation}
  \label{eq:size_of_cd_lower_bound}
  |\CC_{d; \alpha}| \geq
  \binom{|\irrpolys_d'|}{k} \cdot \frac{|G|}{q^{dk + r^2 + O(r) + O(1)}}.
\end{equation}
The number of monic irreducible polynomials of degree $d$ over $\F_q$
is $q^d/d - O(q^{d/2}/d)$, while the number of $*$-symmetric polynomials of degree $d$ is at most $q^{d/2}$, so
\[
  |\irrpolys_d'|
  = (q^d / d - O(q^{d/2})) / 2
  \geq c q^d/d.
\]
By the hypothesis $q^d > Cn$ this is at least $k$, and in fact
\[
  \binom{|\irrpolys_d'|}{k}
  \geq \pfrac{cq^d}{dk}^k
  \geq q^{dk} \exp\br{-O(k \log n)},
\]
so
\[
  \frac{|\CC_{d; \alpha}|}{|G|} \geq \exp\br{-O(d^2 \log q) - O(d^{-1} n \log n)}.
\]
This proves the proposition.
\end{proof}

\subsection{Collecting results}

We now collect the results from the previous sections
to conclude that with high probability as three random elements from $G$ are chosen uniformly at random,
there is a short word in these elements that belongs to $\MM$.

\begin{theorem} \label{thm:reaching_m_3}
There are constants $c, C > 0$ so that the following holds.
Let $G = \Cl_n(q)$,
where $\log q < c n \log^{-2} n$.
Let $x,y,z$ be elements of $G$ chosen uniformly at random.
Let $M$ be the event that there exists a word $w \in F_3$
of length at most $n^{C \log q}$
such that $w(x,y,z) \in \MM$.
Then
\[
  \P_{x,y,z}(M) \geq 1 - e^{- c n}.
\]
\end{theorem}

\begin{proof}
By Corollaries~\ref{cor:expected_character_bound_exponential}
and \ref{cor:xw(y,z)-trick}
there are constants $c_1, c_2 > 0$ and $C_1, C_2$ such that the following holds. Let $\ell = \floor{c_1 n / 2}$, and let $E$ be the event that every word $u \in F_2$ of length at most $\ell$ satisfies $x u(y, z) \notin \CC_d$. Then
\begin{align} \label{eq:prob_E_upper_bound}
  \P(E)
  &\leq \max_{\alpha \in G^\ab} \frac{|G'|}{|\CC_{d;\alpha}|} (q^{-c_1 n} + 2^{-c_1 2 \ell})\\
  &\leq \exp(C_1 d^2 \log q + C_1 d^{-1} n \log n - c_2 n),
\end{align}
provided $q^d > C_2 n$. Take $d \sim C_3 \log n$ for a constant $C_3$. If $\log q < cn / \log^2 n$ for a sufficiently small constant $c$ so that $c, C_3$ satisfy $C_1 C_3^2 c + C_1/C_3 - c_2 < - c$, then $\P(E) \leq e^{-cn}$.

On the other hand suppose $E$ fails, i.e., suppose there is a word $u$ of length at most $c_1 n$ such that $x u(y, z) \in \CC_d$. Let $w \in F_3$ be the word
\[
  w = (\xi_1 u(\xi_2, \xi_3))^{\kappa(q^d-1)}.
\]
The length of $w$ is at most
\[
 \kappa (q^d - 1) (1 + c_1 n / 2) \leq n^{C \log q},
\]
and
\[
  w(x, y, z) = (x u(y, z))^{\kappa(q^d - 1)} \in \MM.
\]
Hence $E^c \subset M$. This completes the proof.
\end{proof}

This completes the proof of Theorem~\ref{thm1:finding-a-transvection}.

If we are allowed $q^C$ random generators, we can reach the set $\MM$ using shorter words.

\begin{theorem} \label{thm:reaching_m_k}
There are constants $c, C > 0$ so that the following holds.
Let $G = \Cl_n(q)$,
where $n > C$.
Let $x_0, x_1, \dots, x_k$ be elements of $G$ chosen uniformly at random, where $k > q^C$.
Let $M$ be the event that there exists a word $w \in F_{k+1}$
of length at most $q^2 n^C$
such that $w(x_0, \dots, x_k) \in \MM$.
Then
\[
  \P_{x_0, \dots, x_k}(M) \geq 1 - q^{- c n}.
\]
\end{theorem}

\begin{proof}
Follow the proof of the previous theorem, replacing
$u \in F_2$ with $u \in F_k$.
Since $\log k > C \log q$,
we can replace \eqref{eq:prob_E_upper_bound} with the bound
\begin{align}
  \P(E)
  &\leq \max_{\alpha \in G^\ab} \frac{|G'|}{|\CC_{d;\alpha}|} q^{- c_2 n} \\
  &\leq \exp(C_1 d^2 \log q + C_1 d^{-1} n \log n - c_2 n \log q),
\end{align}
provided $q^d > C_2 n$. Take $d = \max(\ceil{C_3 \log n / \log q}, 2)$ for sufficiently large $C_3$. As long as $n > C$ we find $\P(E) \leq q^{-cn}$. Note that $q^d \leq q^2 n^C$ in this case.
The rest of the argument is the same.
\end{proof}

\section{Closed trajectories with only one coincidence}
\label{sec:closed_trajectories_one_coincidence}

A trajectory is \emph{closed} if $v^\ell = v^0$. In Section~\ref{sec:trace-method} we will need to understand the structure of closed trajectories with only one coincidence. More generally the joint trajectory of an $r$-tuple $(v_1, \dots, v_r)$ is called closed if each individual trajectory is closed, and we will need to understand the structure of closed joint trajectories with only one coincidence in each individual trajectory. We begin with the single-trajectory case, for motivation.

\begin{lemma} \label{lem:one-coincidence}
Assume $w$ is nontrivial and cyclically reduced. Suppose the trajectory $v^0, \dots, v^\ell$ is closed, and suppose there is only one coincidence, at step $t$ say. Then
\[
  w = (w_d \cdots w_1)^{\ell / d}, \qquad\text{where}~d = \gcd(t, \ell).
\]
In particular if $w$ is not a proper power then $t = \ell$.
\end{lemma}

\begin{proof}
Let
\[
  \tilde w = \cdots w_1 w_\ell \cdots w_1
\]
be the left-infinite $\ell$-periodic extension of $w$. Since $v^\ell = v^0$, the trajectory of $v$ under $\tilde w$ (defined in the obvious way) is just the $\ell$-periodic extension of $v^0, \dots, v^\ell$, and still there is only one coincidence, at step $t$. The choices at steps $1, \dots, t$ are free and all subsequent choices are forced (as in the proof of Lemma \ref{lem:at-least-one-coincidence}).
We claim that $\tilde w$ is in fact $\gcd(t, \ell)$-periodic, and it suffices to prove that it is $t$-periodic.

Since the choices at steps $1, \dots, t - 1$ are free and not coincidences,
the choice at step $t$ is a coincidence,
and all subsequent choices are forced,
the vectors $v^0, \dots, v^{t-1}$ are linearly independent and the whole trajectory is contained in their span. In particular
\begin{equation} \label{v^t}
  v^t = a_0 v^0 + \cdots + a_{t-1} v^{t-1} \qquad (a_0, \dots, a_{t-1} \in \F_q).
\end{equation}
Given that step $t+1$ is forced, we must have $v^i \in D_{w_{t+1}}^{t+1}$ for each $i$ such that $a_i \neq 0$.
Thus either $w_{t+1} = w_{i+1}$ or $w_{t+1} = w_i^{-1}$ ($i > 0$).
Similarly,
\[
    v^{\ell - 1} = b_0 v^0 + \cdots + b_{t-1} v^{t-1} \qquad (b_0, \dots, b_{t-1} \in \F_q),
\]
and $v^\ell = v^0$ is forced.
Since $w_\ell \neq w_1^{-1}$, we must have $w_\ell = w_t$ and $a_0 \neq 0$ (see Remark~\ref{rem:r=1} for more details).
Therefore
\[
  w_{t+1} = w_1.
\]

Consider now the trajectory of $v^1$ under
\[
  \tilde w' = \tilde w w_1^{-1} = \cdots w_3 w_2.
\]
The trajectory is just $v^1, v^2, \dots, v^\ell, v^0, v^1, \dots$. By \eqref{v^t} and $a_0 \neq 0$, $v^1, \dots, v^t$ are linearly independent, and, for every letter $\xi$,
\[
  \sp \{ v^i \mid 0 < i \leq t, v^i \in D_\xi^{t+1} \} = \sp \{ v^i \mid 0 \leq i < t, v^i \in D_\xi^{t} \}.
\]
Therefore the trajectory of $v^1$ also has just one coincidence, again at step $t$ (when $v^{t+1}$ is chosen). Therefore by the same argument we must have $w'_{t+1} = w'_1$, or
\[
  w_{t+2} = w_2.
\]
Repeating this argument as many times as necessary proves that $\tilde w$ is $t$-periodic, as claimed.
\end{proof}

\begin{remark} \label{rem:r=1}
If $t = \ell$, we must have $a_0 = 1$ and all other $a_i = 0$. The general case $t < \ell$ is more complicated, but we can still describe the possibilities. From \eqref{v^t}, because step $t+1$ is forced we must have
\begin{equation} \label{v^{t+1}}
  v^{t+1} = \bar{w_{t+1}} v^t = \sum_{i=0}^{t-1} a_i v^{i \pm 1},
\end{equation}
the signs depending on whether $w_{t+1} = w_{i+1}$ or $w_{t+1} = w_i^{-1}$ ($a_i \neq 0$). At the next step,
\[
  v^{t+2} = \sum_{i=0}^{t-1} a_i v^{i \pm 1 \pm 1},
\]
and so on. We make a few observations:
\begin{enumerate}
  \item The vectors $v^{i \pm 1}$, etc, obey a no-crossing rule: we cannot have
  \begin{align*}
    v^i &\xrightarrow{w_{s+1}} v^{i+1}, \\
    v^{i+1} &\xrightarrow{w_{s+1}} v^i,
  \end{align*}
  as then we would have both $w_{s+1} = w_{i+1}$ and $w_{s+1} = w_{i+1}^{-1}$, for some $i$.
  \item Similarly, there is a no-meeting rule: we cannot have
  \begin{align*}
    v^i &\xrightarrow{w_{s+1}} v^{i+1}, \\
    v^{i+2} &\xrightarrow{w_{s+1}} v^{i+1},
  \end{align*}
  as then we would have both $w_{s+1} = w_{i+1}$ and $w_{s+1} = w_{i+2}^{-1}$, but the expression for $w$ is supposed to be reduced.
  \item Finally, there is a time-consistency rule: we cannot have
  \[
    v^i \xrightarrow{w_{s+1}} v^{i+1} \xrightarrow{w_{s+2}} v^i,
  \]
  as then we would have $w_{s+1} = w_{i+1}$ and $w_{s+2} = w_{i+1}^{-1}$, but again the expression for $w$ is supposed to be reduced; nor could we have
  \[
    v^i \xrightarrow{w_{s+1}} v^{i-1} \xrightarrow{w_{s+2}} v^i,
  \]
  as then we would have $w_{s+1} = w_i^{-1}$ and $w_{s+2} = w_i$.
\end{enumerate}
Since $a_0 \neq 0$ and $w_{t+1} v^0 = v^1$, the only resolution is that
\[
  v^{t+s} = \sum_{i=0}^{t-1} a_i v^{i + s}
\]
for all $s \geq 0$ (extending $\ell$-periodically). In other words, the sequence $(v^s)$ in $\sp\{v^0, \dots, v^{t-1}\}$ corresponds with the sequence $(X^s)$ in $\F_q[X] / (f)$, where
\[
  f = X^t - a_{t-1} X^{t-1} - \cdots - a_0 X^0.
\]
Since $v^\ell = v^0$ we must have
\[
  f \mid X^\ell - 1.
\]
Conversely, if $f$ is a divisor of $X^\ell - 1$, and if the period of $w$ divides $t$ and $i - i'$ whenever $a_i \neq 0$ and $a_{i'} \neq 0$, then a one-coincidence trajectory of this type exists.
\end{remark}

We now consider closed joint trajectories with only one coincidence in each individual trajectory.
The following lemma generalizes Lemma~\ref{lem:one-coincidence}.

\begin{lemma} \label{lem:r-coincidences}
Assume $w$ is nontrivial and cyclically reduced.
Let $v_1, \dots, v_r \in V$ be linearly independent.
Suppose the joint trajectory of $v_1, \dots, v_r$ is closed.
Suppose there is just one coincidence in each individual trajectory, and suppose the coincidence in the trajectory of $v_i$ occurs at step $(t_i, i)$.
Then
\[
  w = (w_d \cdots w_1)^{\ell / d}, \qquad \text{where}~d = \gcd(t_1, \dots, t_r, \ell).
\]
In particular if $w$ is not a proper power then $t_i = \ell$ for each $i$.
\end{lemma}

\begin{proof}
As in the proof of Lemma~\ref{lem:one-coincidence}, let $\tilde w$ be the left-infinite $\ell$-periodic extension of $w$, and note that the trajectory of $v_1, \dots, v_r$ under $\tilde w$ is just the $\ell$-periodic extension of the trajectory under $w$, and there are no further free choices.

The choice at step $(t, i)$ must be free for $t \leq t_i$ and forced for $t > t_i$. Therefore the vectors $(v_i^t)_{1 \leq i \leq r, 0 \leq t < t_i}$ are linearly independent and the whole trajectory is contained in their span. Since there is a coincidence at step $(t_i, i)$, we have
\begin{equation} \label{v_i^{t_i}}
  v_i^{t_i} = \sum_{(t, j) \prec (t_i, i)} a_{itj} v_j^t \qquad (a_{itj} \in \F_q),
\end{equation}
where $a_{itj} = 0$ whenever $t \geq t_j$ (and $(t, j) \prec (t_i, i)$ means $t < t_i$ or $t = t_i$ and $j < i$, as in Subsection~\ref{subsec:trajectories}).
Let $A_0$ be the $r \times r$ matrix
\[
  A_0 = (a_{i0j} : 1 \leq i,j \leq r).
\]
The matrix $A_0$ must be nonsingular, for otherwise we could not have $(v_1^\ell, \dots, v_r^\ell) = (v_1^0, \dots, v_r^0)$. In particular, for each $i$ there is some $j$ such that $a_{i0j} \neq 0$. Since step $(t_i+1, i)$ is forced, the value of $w_{t_i+1} v_j^0$ must be known; hence
\[
  w_{t_i+1} = w_1.
\]

Consider the joint trajectory of $(v_1^1, \dots, v_r^1)$ under
\[
  \tilde w' = \tilde w w_1^{-1} = \cdots w_3 w_2,
\]
which is just $(v_i^t)_{1 \leq i \leq r, t \geq 1}$. Since $A_0$ is nonsingular, we have
\[
  \sp \{ v_i^t : 1 \leq i \leq r, 1 \leq t \leq t_i\}
  = \sp \{ v_i^t : 1 \leq i \leq r, 0 \leq t \leq t_i - 1\}.
\]
Therefore the vectors $(v_i^t)_{1 \leq i \leq r, 1 \leq t \leq t_i}$ are linearly independent, and the joint trajectory of $(v_1^1, \dots, v_r^1)$ under $\tilde w'$ has the same behaviour as that of $(v_1, \dots, v_r)$ under $\tilde w$: the trajectory of $v_i^1$ has just one coincidence, at step $(t_i, i)$ (when $v_i^{t_i+1}$ is chosen). Therefore by the same argument $w'_{t_i+1} = w'_1$, or
\[
  w_{t_i+2} = w_2.
\]
Repeating the argument as many times as necessary, we conclude that the period of $\tilde w$ divides $t_i$ for each $i$.
\end{proof}

\begin{remark}
The discussion in Remark~\ref{rem:r=1} generalizes too.
From \eqref{v_i^{t_i}} and forcedness, we have
\begin{equation} \label{v_i^{t_i+1}}
  v_i^{t_i + 1}
  = \sum_{(t, j) \prec (t_i, i)} a_{itj} v_j^{t \pm 1}
  \qquad (i \in \{1, \dots, r\}),
\end{equation}
where the signs are chosen depending on whether $w_{t+1} = w_1$ or $w_t = w_1^{-1}$. The latter case can arise only for $t > 0$, so no $v_j^0$ can appear in this expression. Hence \eqref{v_i^{t_i+1}} is the analogue of \eqref{v^{t+1}} for the joint trajectory of $(v_1^1, \dots, v_r^1)$. As before there are no-crossing, no-meeting, and time-consistency rules for the indices $t$ such that $a_{itj}\neq 0$ for some $i, j$, so in fact we can never have $v_j^{t-1}$.

We conclude that
\[
  v_i^{t_i+s} = \sum_{(t, j) \prec (t_i, i)} a_{itj} v_j^{t+s}
\]
for all $s \geq 0$, and hence the trajectory of $(v_1^s, \dots, v_r^s)$ corresponds with the trajectory of $(Z^s X_1, \dots, Z^s X_r)$ in the $\F_q[Z]$-module $(\F_q[Z] X_1 \oplus \cdots \oplus \F_q[Z] X_r) / \langle f_1, \dots, f_r \rangle$, where
\[
  f_i = Z^{t_i} X_i - \sum_{(t, j) \prec (t_i, i)} a_{itj} Z^t X_j,
\]
and we must have
\[
  \langle (Z^\ell -  1) X_1, \dots, (Z^\ell -1) X_r \rangle \subset \langle f_1, \dots, f_r \rangle.
\]
Write $f_i = \sum_j p_{ij} X_j$ for some $p_{ij} \in \F_q[Z]$ and let $F = (p_{ij} : 1 \leq i, j \leq r)$. Then there must exist a matrix $E \in \M_r(\F_q[Z])$ with
\[
  (Z^\ell - 1) I = E F.
\]
This is possible if and only if $\det F$ divides $Z^\ell - 1$.
\end{remark}

\section{Expansion in low-degree representations} \label{sec:trace-method}

We turn now to the proof of Theorem~\ref{thm2:schreier-expansion}.
We again consider the action of $G = \Cl_n(q)$ on linearly independent $r$-tuples of vectors,
and we again consider trajectories under the action of a fixed word $w \in F_k$, much as in Section~\ref{sec:prob-small-support}. The difference is mainly one of parameter regime. In Section~\ref{sec:prob-small-support} we considered $r$-tuples with $r$ as large as $cn$ for constant $c$, and we were satisfied with somewhat crude bounds. In this section we consider $r = O(1)$, and we seek sharper bounds. Our aim is to show that, in an orbit of $G$ of size $N$, the probability that a trajectory under a given word closes is close to $1/N$, with a small relative error; if we can do this it follows that there is a uniform spectral gap. We begin with the case of $r=1$, which contains most of the key ideas.

\subsection{The defining representation}

Now let $x_1, \dots, x_k \in G$ be chosen uniformly at random. Let $\bar w = w(x_1, \dots, x_k)$. Let $v \in V \setminus\{0\}$. Let $N = |G v|$. By Witt's lemma (Lemma~\ref{lem:witt}), $N$ is the number of $u \in V \setminus\{0\}$ such that $Q(u) = Q(v)$. Thus, by Lemma~\ref{lem:subspace-counting-lemma}, $N = q^n/q_0 + O(q^{n/2})$. More generally, if $U \leq V$ is a subspace of dimension $d$ then
\[
  |Gv \cap U| = q^d/q_0 + O(q^{n/2}).
\]

\begin{lemma} \label{lem:wv=v}
Assume $w$ is nontrivial and not a proper power.
Assume $\ell < n/4$.
Then
\[
  \P(\bar w v = v)
  \leq \frac1N \br{
    1 + O(q^{2\ell - n/2})
  }.
\]
\end{lemma}
\begin{proof}
By Lemma~\ref{lem:barw-coordinate-free} we may also assume that $w$ is cyclically reduced, as replacing $w$ by its cyclic reduction can only decrease its length. In this case Lemma~\ref{lem:one-coincidence} implies that the event that $\bar w v = v$ is contained in the union of the following two events:
\begin{itemize}
  \item[$E_1$:] the trajectory $v^0, \dots, v^\ell$ has exactly one coincidence, occuring at step $\ell$, and $v^\ell = v^0$,
  \item[$E_2$:] the trajectory $v^0, \dots, v^\ell$ has at least two coincidences.
\end{itemize}

We can bound the probability of $E_2$ using Lemma~\ref{lem:free-choices}. Suppose there is a free choice at step $t \leq \ell$. There are $t$ previous vectors, so the probability of a coincidence, conditional on previous steps, is bounded by
\[
  \frac{q^t}{q^{n-t} - q^{t-1} - q^{n/2}}.
\]
Similarly, the conditional probability of a coincidence at a later step $t'$ is bounded by
\[
  \frac{q^{t'-1}}{q^{n-t'} - q^{t'-1} - q^{n/2}}.
\]
Summing over $t < t' \leq \ell$, we find, using $\ell < n/4$,
\[
  \P(E_2)
  \leq \sum_{1\leq t < t' \leq \ell} \frac{q^{t+t'-1}}{(q^{n-\ell} - q^{\ell - 1} - q^{n/2})^2}
  \ll q^{4\ell - 2n}
  \leq q^{4 \ell - n} / N
  < q^{2 \ell - n/2} / N.
\]

Hence we may focus on the event $E_1$. In the linear case ($\Cl=\SL$), $v^\ell$ is chosen uniformly at random outside a linear subspace of dimension at most $\ell-1$, so the probability of $E_1$ is bounded by
\[
  \frac1{q^n - q^{\ell - 1}} = \frac1N \br{
    1 + O(q^{\ell - n})
  }.
\]
This completes the proof in this case.

In general, the situation is complicated by form conditions, as previous choices may significantly impact the probability that $v^\ell = v^0$, even if there were no previous coincidences.

Let $\xi = w_\ell$. The choice of $v^\ell$ is subject to one linear constraint for every occurence of $\xi = w_\ell$ as $w_t$ or $w_{t+1}^{-1}$ for some $t < \ell$.
Each such occurence is the end of a maximal subword matching a prefix $u = w_\ell \cdots w_{\ell-s+1}$ of $w$,
forward in the case $\xi = w_t$ and backward in the case $\xi = w_{t+1}^{-1}$ (see Figure~\ref{fig:subwordu}).
Write $s = s(t)$ and $u = u(t)$.
Define
\begin{align}
  T_1 &= \{t < \ell : \xi = w_{t+1}^{-1}\}, \\
  T_2 &= \{t < \ell : \xi = w_t, \ell - s(t) > t\}, \\
  T_3 &= \{t < \ell : \xi = w_t, \ell - s(t) \leq t\}.
\end{align}
Note that, for $t \in T_1$, we must have $t + s < \ell - s$, because $w_\ell \cdots w_{t+1}$ is reduced.
In the $\xi = w_t$ case it is possible that the subword overlaps (or is adjacent to) the matching prefix, and the division into $T_2$ and $T_3$ reflects this possibility.

\begin{figure}
  \begin{align*}
    w &= \underbrace{w_\ell \cdots w_{\ell-s+1}}_u \cdots \underbrace{w_{t+s} \cdots w_{t+1}}_{u^{-1}} \cdots
    && (t \in T_1)  \\
    w &= \underbrace{w_\ell \cdots w_{\ell-s+1}}_u \cdots \underbrace{w_t \cdots w_{t-s+1}}_u \cdots
    && (t \in T_2)  \\
    w &= \mathrlap{\underbrace{\phantom{w_\ell\cdots w_t \cdots w_{\ell-s+1}}}_u}w_\ell \cdots \overbrace{w_t \cdots w_{\ell-s+1} \cdots w_{t-s+1}}^u \cdots
    && (t \in T_3)
  \end{align*}
  \caption{The word $w$ and one of its maximal subwords matching a prefix $u$. Each occurence of the letter $w_\ell$ or $w_\ell^{-1}$ is the end of one such subword. In the $w_\ell = w_{t+1}^{-1}$ case we must have $t + s < \ell - s$.}
  \label{fig:subwordu}
\end{figure}

The choice of $v^\ell$ at step $\ell$ is constrained by the linear conditions
\begin{align}
  f(v^\ell, v^t) &= f(v^{\ell-s}, v^{t+s}) && (t \in T_1) \\
  f(v^\ell, v^t) &= f(v^{\ell-s}, v^{t-s}) && (t \in T_2 \cup T_3)
\end{align}
(where $s = s(t)$). We need to determine whether $v^0$ is in this affine subspace. Obviously this is the case if and only if
\begin{align}
  f(v^0, v^t) &= f(v^{\ell-s}, v^{t+s}) && (t \in T_1) \\
  f(v^0, v^t) &= f(v^{\ell-s}, v^{t-s}) && (t \in T_2 \cup T_3).
\end{align}
Write $C_t$ for this condition.
For $t \in T_1 \cup T_2$, the truth or falsity of $C_t$ is determined at step $\ell - s$, because $\ell - s > t + s$ in the $t \in T_1$ case and $\ell - s > t$ in the $t \in T_2$ case. The condition is not determined before step $\ell - s$ by maximality of $u(t)$.
For $t \in T_3$, $C_t$ is settled at step $t$, because $t \geq \ell - s$. The condition is not settled before step $t$ because $w_t = w_\ell \neq w_1^{-1}$ (since $w$ is cyclically reduced).

Note that we may have $\ell - s = t$ for $t \in T_3$: this is the case in which the subword is adjacent to the prefix (see Figure~\ref{fig:uu-case}). In this case the condition $C_t$ is
\[
  f(v^0, v^t) = f(v^t, v^{t-s}).
\]
However, we cannot have also $t - s = 0$, for then we would have $w = u^2$. Hence, by linear independence of $v^0, \dots, v^{t-1}$, still the condition $C_t$ is settled at step $t$ and not before. Note, however, if $G$ unitary then $C_t$ is linear only over $\F_{q_0}$ (because the form $f$ is only sesquilinear).

\begin{figure}
  \[
  w = \underbrace{w_\ell \cdots w_{\ell-s+1}}_u \underbrace{w_t \cdots w_{t-s+1}}_u \cdots
  \]
  \caption{The case $t = \ell - s \in T_3$. In this case we must have $t - s > 0$, or else $w = u^2$.}
  \label{fig:uu-case}
\end{figure}

There is a case that may arise in which the various conditions $C_{t'}$ settled at a given step $t$ are not independent.
This is the case in which $t \in T_3$ and $t = \ell - s'$ for some $t' \in T_2$, where $s' = s(t')$, and $t' - s' = 0$ (see Figure~\ref{fig:uu'}).
Let $T_4$ be the set of such steps $t$ and let $T_3' = T_3 \setminus T_4$.
If $t \in T_4$ then we have an overdetermined pair of conditions
\begin{align}
    f(v^0, v^{t'}) &= f(v^t, v^0) && (C_{t'})\\
    f(v^0, v^t) &= f(v^{\ell-s}, v^{t-s}) && (C_t).
\end{align}
This system is consistent if and only if
\[
    f(v^{t'}, v^0)
    = f(v^{\ell-s}, v^{t-s}).
\]
For $t \in T_4$ let us redefine $C_t$ to be this reduced condition.
Certainly $t - s < \ell - s$, and if $t' = \ell - s$ then $wu' = u'w$, so $w$ is a proper power, contrary to hypothesis.
Hence $C_t$ is settled at step $\ell - s \leq t$.

\begin{figure}
    \[
        w = \mathrlap{\, \underbrace{\phantom{w_\ell\cdots w_t w_{t+1} \cdots w_{\ell-s+1}}}_u \phantom{w_{\ell-s}\cdots w_{t-s+1} w_{t-s} \cdots} \, \underbrace{\phantom{w_{t'} \cdots w_1}}_{u'}}
        \overbrace{w_\ell \cdots w_{t+1}}^{u'} \overbrace{w_t \cdots w_{\ell-s+1} w_{\ell-s} \cdots w_{t-s+1}}^u w_{t-s} \cdots w_{t'} \cdots w_1
    \]
    \[
        w = \mathrlap{\, \underbrace{\phantom{w_\ell\cdots w_t w_{t+1} \cdots w_{\ell-s+1}}}_u \underbrace{\phantom{w_{\ell-s}\cdots w_1}}_{u'}}
        \overbrace{w_\ell \cdots w_{t+1}}^{u'} \overbrace{w_t \cdots w_{\ell-s+1} w_{\ell-s} \cdots w_{1}}^u
    \]
  \caption{
    The case $t \in T_4 \subset T_3$.
    Here $t = \ell - s'$ for some $t' \in T_2$ with $t' - s' = 0$.
    If $t' = \ell - s$ then $w$ must be a proper power.
  }
  \label{fig:uu'}
\end{figure}

Now consider any step $t \in \{1, \dots, \ell-1\}$, and consider all those conditions $C_{t'}$ which are settled at step $t$. These conditions are $C_{t'}$ for $t' \in T_1 \cup T_2 \cup T_4$ such that $\ell - s' = t$, as well as $C_t$ if $t \in T_3'$, i.e.,
\begin{align}
  f(v^t, v^{t' + s'}) &= f(v^0, v^{t'}) && (t' \in T_1, \ell - s' = t) \\
  f(v^t, v^{t' - s'}) &= f(v^0, v^{t'}) && (t' \in T_2, \ell - s' = t) \\
  f(v^t, v^{t' - s'}) &= f(v^{t''}, v^0) && (t' \in T_4, \ell - s' = t) \\
  f(v^t, v^0) &= f(v^{t-s}, v^{\ell-s}) && (\text{if}~t \in T_3'.)
\end{align}
We claim that these affine conditions for $v^t$ are independent, and it suffices to demonstrate that the indices
$t' + s'$ ($t' \in T_1, \ell - s' = t$),
$t' - s'$ ($t' \in T_2 \cup T_4$, $\ell - s' = t$),
and $0$ if $t \in T_3'$ are all distinct.
Since $s' = \ell - t$ is a constant, the indices $t'+s'$ are all distinct for $t' \in T_1$, as are the indices $t' - s'$ for $t' \in T_2 \cup T_4$.
Moreover we cannot have $t_1 + s_1 = t_2 - s_2$ for $t_1 \in T_1$ and $t_2 \in T_2 \cup T_4$ with $\ell - s_1 = \ell - s_2 = t$,
because then we would have $w_{t_1+s_1} = w_{t_2-s_2+1}^{-1} = w_{t_1+s_1+1}^{-1}$, in contradiction with the reducedness of $w$.
If $t' - s' = 0$ for some $t' \in T_2$ then $t \in T_4$ by definition, so $t \notin T_3'$.
Finally, if $t' \in T_4$ then we cannot have $t' - s' = 0$ unless $w$ is a proper power, as discussed.

Hence, by linear independence of $v^0, \dots, v^{t-1}$, the $h$ (say) conditions $C_{t'}$ settled at step $t$ consist of $h$ independent affine linear conditions for $v^t$, or, in the unitary case, if $t = \ell - s \in T_3$, $2h$ independent affine linear conditions over $\F_{q_0}$. Suppose $v^t$ is drawn from a subspace of codimension $d$ ($d$ is the number of previous occurences of $w_t$ or $w_{t}^{-1}$). Then, by Lemma~\ref{lem:subspace-counting-lemma} and Lemma~\ref{lem:free-choices}, the probability that all these conditions are satisfied, conditional on the past trajectory $v^0, \dots, v^{t-1}$, is
\begin{align}
  \frac{q^{n-d-h}/q_0 + O(q^d + q^{n/2})}
  {q^{n-d}/q_0 + O(q^d + q^{n/2})}
  &= q^{-h} \br{1 + O(q^{h+d-n/2} q_0)} \\
  &= q^{-h} \br{1 + O(q^{\ell + t - n/2})} \label{eq:single-t-bound}
\end{align}
(in the second line we used $h < \ell$, $d < t$, and $q_0 \leq q$).

Suppose $H = |T_1| + |T_2| + |T_3'| + |T_4|$ (i.e., let $H+1$ be the number of appearances of $w_\ell$ or $w_\ell^{-1}$ in $w$). Taking the product of \eqref{eq:single-t-bound} over all $t$, the probability that $C_{t'}$ is satisfied for every $t' \in T_1 \cup T_2 \cup T_3' \cup T_4$ is
\[
  q^{-H}\br{
  1 + O(q^{2\ell-n/2})
  }.
\]

The conditions $C_t$ are prequisite to the event $v^\ell = v^0$. If all these conditions are satisfied, then at step $\ell$ the vector $v^\ell$ is drawn from an affine subspace of codimension $H$ which includes $v^0$. Note also that $Q(v^{\ell-1}) = Q(v^0)$. Hence, from Lemma~\ref{lem:free-choices},
\[
  \P(v^\ell = v^0 \mid v^0, \dots, v^{\ell-1})
  = \frac{1}{q^{n-H} / q_0 - O(q^H) - O(q^{n/2})}.
\]
Hence the overall probability of $E_1$ is bounded by
\begin{align*}
  \frac{q^{-H}\br{1 + O(q^{2\ell - n/2})}}{q^{n-H} / q_0 - O(q^H) - O(q^{n/2})}
  &= \frac{1}{q^n/q_0} \br{
    1 + O(q^{2\ell - n/2})
  } \\
  &= \frac{1}{N} \br{
    1 + O(q^{2\ell - n/2})
  }.
\end{align*}
Thus in all cases the error is bounded as claimed.
\end{proof}

\begin{remark}
  In the linear case, the hypothesis that $w$ is not a proper power is needed only to ensure that the event $v^\ell = v^0$ is contained in $E_1 \cup E_2$; we do not need the hypothesis in order to bound $\P(E_1)$ or $\P(E_2)$.
  By contrast, at least in the orthogonal case, we do need this hypothesis in order to bound $\P(E_1)$ satisfactorily, so at least some of the complexity of the above proof is necessary.
  Suppose $G = \GO_n(q)$ and $w = u^2$ for some word $u$ of length $\ell / 2$. Then the choice of $v^\ell$ is constrained by
  \[
    f(v^\ell, v^{\ell/2}) = f(u v^{\ell/2}, u v^0) = f(v^{\ell/2}, v^0) = f(v^0, v^{\ell/2}).
  \]
  Hence $v^\ell$ is always restricted to an affine hyperplane that includes $v^0$, so the probability that $\bar w v = v$ will be at least approximately $q/N$, even conditionally on there being only one coincidence.
\end{remark}

\begin{remark}
On the other hand, it is usually possible to cyclically rotate $w$ so that much of the complexity in the previous proof disappears. For example, if $w$ can be cyclically rotated so that it has no square prefix, then, after such a rotation, $T_3 = \emptyset$. Not every non-proper-power has this property,\footnote{e.g., $xxyxxyxxyxy$} but almost all words do.
\end{remark}

We can now prove that the permutation action of uniformly random $x_1, \dots, x_k \in G$ on an orbit $Gv \subset V$ has a uniform spectral gap.
Assume $v \neq 0$.
As usual let $\adjcy$ be the normalized adjacency operator
\[
  \adjcy = \frac{1}{2k} \sum_{i=1}^k (x_i + x_i^{-1})
\]
acting on $\C[Gv]$,
and let $1 = \lambda_1 \geq \lambda_2 \geq \cdots \geq \lambda_N$ be the spectrum.
Let $\lambda = \max(\lambda_2, -\lambda_N)$.
Then, for even $\ell$,
\[
  1 + \lambda^\ell
  \leq \tr \adjcy^\ell
  = \E_w |\{u \in Gv : \bar w u = u\}|,
\]
where $w$ is the result of a simple random walk of length $\ell$ in $F_k$. Let $\pp \subset F_k$ be the set of proper powers $w^m$ ($w \in F_k, m \geq 2$). Then
\begin{align*}
  \E \lambda^\ell
  &\leq \E_{x_1, \dots, x_k} \E_w |\{ u \in Gv : \bar w u = u \}| - 1\\
  &= \E_w \left(\P(\bar w v = v) - \frac1N \right) N\\
  &\leq \P(w\in\pp) N + \max_{w \notin \pp, |w| \leq \ell} \left(\P(\bar w v = v) - \frac1N \right) N.
\end{align*}
By \cite[Lemma~2.6]{FJRST},
\[
  \P(w \in \pp) \ll \ell \pfrac{2k - 1}{k^2}^{\ell / 2} \ll k^{-c\ell}.
\]
By Lemma~\ref{lem:wv=v},
\[
  \max_{w \notin \pp, |w| \leq \ell} \P(\bar w v = v) \leq \frac{1}{N}\br{1 + O(q^{2\ell - n/2})},
\]
provided $\ell < n/4$.
Hence
\[
  \E \lambda^\ell \ll k^{-c\ell} q^n + q^{2\ell - n/2}.
\]
Take $\ell \sim n/5$. If $\log k / \log q$ is sufficiently large then
\[
  \E \lambda^\ell \leq q^{-c' \ell}.
\]
Hence, by Markov's inequality,
\[
  \P(\lambda \geq q^{-c'/2}) = \P(\lambda^\ell \geq q^{-c'\ell/2}) \leq q^{c'\ell/2} \E \lambda^\ell \leq q^{-c'\ell/2} \leq q^{-c''n},
\]
so almost surely $\lambda < q^{-c'/2}$.

\subsection{The action on \texorpdfstring{$r$}{r}-tuples}

We now generalize the argument of the previous subsection to $r$-tuples of vectors, where $r$ is bounded. It will be convenient to use the following notation. For $v, v' \in V^r$, let $f(v, v')$ denote the $r \times r$ matrix
\[
  f(v, v')_{ij} = f(v_i, v'_j).
\]
Define also
\[
  Q(v)_i = Q(v_i).
\]

Let $v = (v_1, \dots, v_r) \in V^r$, where $v_1, \dots, v_r \in V$ are linearly independent. Let $N = |Gv|$. By Witt's lemma, $N$ is the number of $v' \in V^r$ with $v'_1, \dots, v'_r$ linearly independent such that $f(v, v) = f(v', v')$ and $Q(v) = Q(v')$. In the linear case,
\begin{align*}
  N
  &= (q^n - 1) (q^n - q) \cdots (q^n - q^{r-1})\\
  &= q^{rn} \br{1 - O(q^{-n+r-1})}.
\end{align*}
In the other cases we have, inductively, using Lemma~\ref{lem:subspace-counting-lemma},
\begin{align}
  N
  &= |G(v_1, \dots, v_{r-1})| (q^{n-r+1} / q_0 + O(q^{n/2})) \\
  &= q^{rn - r(r-1)/2} / q_0^r \br{1+ O(q^{-n/2+r-1} q_0)}. \label{eq:size-of-N-r>1}
\end{align}

\begin{lemma} \label{lem:wv=v--r>1}
Assume $w$ is nontrivial and not a proper power.
Assume $\ell r^2 < n/4$.
Then
\[
  \P(\bar w v = v) \leq \frac1N \br{
  1 + O( q^{2 \ell r - n/2} )
  }.
\]
\end{lemma}
\begin{proof}
Again we may assume $w$ is cyclically reduced.
In this case Lemma~\ref{lem:r-coincidences} implies that the event that $\bar w v = v$ is contained in the union of the following two events:
\begin{itemize}
  \item[$E_1$:] the joint trajectory $(v_i^t)$ has exactly one coincidence in each individual trajectory, each occuring at the final step $t=\ell$, and $v_i^\ell = v_i^0$ for each $i$,
  \item[$E_2$:] the joint trajectory $(v_i^t)$ has at least $r+1$ coincidences.
\end{itemize}

Again we can bound the probability of $E_2$ using Lemma~\ref{lem:free-choices}.
Suppose there is a free choice at step $(t, i)$.
There are at most $t r + i - 1 \leq \ell r$ previous vectors, so the conditional probability of a coincidence is bounded by
\[
  \frac{q^{t r + i - 1}}
  {q^{n - \ell r} - q^{\ell r - 1} - q^{n/2}}
  = q^{tr + i - 1 + \ell r - n} \br{1 + O(q^{\ell r - n/2})}.
\]
Hence the probability of $E_2$ is bounded by (summing over all possibilities for $r+1$ coincidences)
\[
  q^{(2 \ell r - n)(r+1)} \br{1 + O(r q^{\ell r - n/2})} \ll q^{(2 \ell r - n)(r+1)}.
\]
Using $N \leq q^{rn}$, this is at most
\[
  q^{2 \ell r (r+1) - n} / N \leq q^{2 \ell r - n/2} / N.
\]

Hence we may focus on the event $E_1$. In the linear case, for each $i$ the vector $v_i^\ell$ is chosen uniformly at random outside a linear subspace of dimension at most $\ell r$, so the probability of $E_1$ is bounded by
\begin{align}
  \pfrac1{q^n - q^{\ell r}}^r
  &= q^{-rn} \br{ 1 + O(r q^{\ell r - n}) } \\
  &= \frac1N \br{ 1 + O(r q^{\ell r - n}) }.
\end{align}
This completes the proof in this case.

As in the previous subsection, the general situation is complicated by form conditions, but fortunately few changes are necessary in the $r > 1$ case.
Let $\xi = w_\ell$.
Assume there are $H+1$ occurences of $\xi$ or $\xi^{-1}$ in $w$,
and consider the $H$ maximal subwords $u$ ending with $\xi$ or $\xi^{-1}$
and matching a proper prefix of $w$,
as in Figure~\ref{fig:subwordu}.
Define $T_1$, $T_2$, and $T_3 = T'_3 \cup T_4$ as before.

The choice of $v^\ell$ at step $\ell$ is constrained by the linear conditions
\begin{align}
  f(v^\ell, v^t) &= f(v^{\ell-s}, v^{t+s}) && (t \in T_1) \\
  f(v^\ell, v^t) &= f(v^{\ell-s}, v^{t-s}) && (t \in T_2 \cup T_3)
\end{align}
(where $s = s(t)$).
For $t \in T_1 \cup T_2 \cup T_3'$ we have a condition $C_t$ defined by
\begin{align}
  f(v^0, v^t) &= f(v^{\ell-s}, v^{t+s}) && (t \in T_1) \\
  f(v^0, v^t) &= f(v^{\ell-s}, v^{t-s}) && (t \in T_2 \cup T_3').
\end{align}
For $t \in T_4$ the condition $C_t$ is the reduced condition
\[
    f(v^{t'}, v^0) =  f(v^{\ell - s}, v^{t-s}).
\]
Conditional on linear independence of $v_i^t$ for $1 \leq i \leq r$ and $t < \ell$, it can be verified exactly as in the $r = 1$ case that the conditions settled at any given step $t < \ell$ are precisely $C_{t'}$ for $t' \in T_1 \cup T_2 \cup T_4$ and $\ell - s' = t$, as well as $C_t$ if $t \in T_3'$, and these conditions are linearly independent.

Suppose at step $t < \ell$ there are $h$ conditions $C_{t'}$ to be settled. Assume first that we are not in the case $t = \ell - s \in T_3'$
(the case in which the subword is adjacent to the prefix, as in Figure~\ref{fig:uu-case}).
Let $d$ be the number of previous occurences of $w_t$ or $w_t^{-1}$. Then, by Lemma~\ref{lem:free-choices}, at step $(t, i)$ the vector $v_i^t$ is drawn from an affine subspace of codimension $d' = dr + i-1$, less a subspace of dimension $d'$, subject to the quadratic condition $Q(v_i^t) = Q(v_i^{t-1})$. Hence, using Lemma~\ref{lem:subspace-counting-lemma}, the probability that $ji$-component of each $C_{t'}$ is satisfied for each $j \in \{1, \dots, r\}$ is
\begin{align}
  \frac{q^{n - d' - hr}/q_0 + O(q^{d'} + q^{n/2})}
  {q^{n-d'}/q_0 + O(q^{d'} + q^{n/2})}
  &= q^{-hr} \br{
    1 + O( q^{hr+d'-n/2} q_0)
  } \\
  &= q^{-hr} \br{
    1 + O( q^{\ell r + (t-1) r + i-1 -n/2})
  } \label{eq:single-t-single-i-bound}
\end{align}
(using $h < \ell$, $d' \leq (t-1)r + i-1$, and $q_0 \leq q$).
Taking the product over all $i$, the probability that each $C_{t'}$ is satisfied after step $t$ is
\begin{equation}
  q^{-hr^2} \br{
    1 + O(q^{\ell r + tr - n/2})
  }. \label{eq:single-t-bound-r>1}
\end{equation}

The case $t = \ell - s \in T_3'$ is slightly different. In this case the $ji$-component of $C_t$ is
\[
  f(v_j^0, v_i^t) = f(v_j^t, v_i^{t-s}).
\]
This condition is settled at step $(t, k)$, where $k = \max(i, j)$.
Hence $2k-1$ components of $C_t$ are settled at step $(t, k)$.
Therefore, in this case, \eqref{eq:single-t-single-i-bound} must be replaced with
\begin{equation}
  q^{-(h-1)r - (2i - 1)} \br{
    1 + O( q^{\ell r + (t-1) r + i-1 -n/2})
  } \label{eq:single-t-single-i-bound-var}.
\end{equation}
Taking the product over all $i$ again gives \eqref{eq:single-t-bound-r>1}.

Taking the product of \eqref{eq:single-t-bound-r>1} over all $t$, the probability that $C_{t'}$ is satisfied for every $t' \in T_1 \cup T_2 \cup T_3' \cup T_4$ is
\begin{equation} \label{eq:prob-all-C_t}
  q^{-Hr^2}\br{
  1 + O(q^{2\ell r - n/2})
  }.
\end{equation}

Finally, if all the conditions $C_t$ are satisfied, then for each $i$ the vector $v_i^\ell$ is drawn from an affine subspace of codimension $Hr+i-1$ which includes $v_i^0$, less a subspace of dimension $Hr+i-1$, subject to the quadratic condition $Q(v_i^\ell) = Q(v_i^{\ell-1}) = Q(v_i^0)$. Hence
\begin{align*}
  \P(v_i^\ell = v_i^0 \mid (v_j^t, (t, j) \prec (\ell, i)))
  &= \frac{1}{q^{n-Hr-i+1} / q_0 - O(q^{Hr+i-1}) - O(q^{n/2})} \\
  &= (q^{n-Hr-i+1}/q_0)^{-1} \br{
  1 + O(q^{Hr+i-1-n/2} q_0)
  }
\end{align*}
Hence the conditional probability that $v^\ell = v^0$ is
\[
  (q^{nr-Hr^2-r(r-1)/2} / q_0^r)^{-1} \br{
  1 + O(q^{(H+1)r - n/2})
  }.
\]
Hence the overall probability of $E_1$ is, multiplying the previous line by \eqref{eq:prob-all-C_t},
\[
  (q^{nr - r(r-1)/2} / q_0^r)^{-1} \br{
    1 + O(q^{2\ell r - n/2})
  }.
\]
Comparing with \eqref{eq:size-of-N-r>1}, this is
\[
  N^{-1} \br{
    1 + O(q^{2\ell r - n/2})
  }.
\]
Thus in all cases the error is bounded as claimed.
\end{proof}

We can now prove that the permutation action of uniformly random $x_1, \dots, x_k \in G$ on an orbit $Gv \subset V^r$ has a uniform spectral gap.
The argument is little different from that in the previous subsection.
We may assume $v_1, \dots, v_r$ are linearly independent, by reducing $r$ if necessary.
Suppose the adjacency operator $\adjcy$ acting on $\C[Gv]$ has spectrum
$1 = \lambda_1 \geq \cdots \geq \lambda_N$.
Let $\lambda = \max(\lambda_2, -\lambda_N)$.
For even $\ell$, let $w$ be the result of a simple random walk of length $\ell$ in $F_k$. Then
\[
  \E \lambda^\ell
  \leq \P(w \in \pp) N
  + \max_{w \notin \pp, |w|\leq \ell} \br{\P(\bar w v = v) - \frac1N} N.
\]
We bound $\P(w \in \pp)$ as before, while by Lemma~\ref{lem:wv=v--r>1} we have
\[
  \max_{w \notin \pp, |w| \leq \ell} \P(\bar w v = v)
  \leq \frac1N \br{
    1 + O(q^{2\ell r - n/2})
  },
\]
provided $\ell r^2 < n/4$.
Hence
\[
  \E \lambda^\ell
  \ll k^{-c \ell} q^{rn} + q^{2 \ell r - n/2}.
\]
Take $\ell \sim n/(5 r^2)$. If $\log k / \log q \geq C r^3$, for a sufficiently large constant $C$, then
\[
  \E \lambda^\ell
  \leq q^{-c' \ell}.
\]
Hence, by Markov's inequality,
\[
  \P(\lambda \geq q^{-c'/2})
  \leq q^{c' \ell / 2} \E \lambda^\ell
  \leq q^{-c'\ell / 2}
  < q^{-c''n/r^2},
\]
so almost surely $\lambda < q^{-c'/2}$, as before.

\subsection{Other low-degree representations}

The result of the final argument of the previous subsection can be expressed as follows.

\begin{theorem}
\label{thm:CVr-spectral-gap}
Let $\C[V^r]_0$ be the orthogonal complement of $\C[V^r]^G$ in $\C[V^r]$.
Let $x_1, \dots, x_k \in G$ be uniform and independent, where $k \geq q^{Cr^3}$ and $r < cn^{1/4}$.
Let $\rho = \rho(\adjcy, \C[V^r]_0)$ be the spectral radius of $\adjcy = \adjcy_{x_1, \dots, x_k}$ acting on $\C[V^r]_0$.
Then
\[
  \P(\rho > q^{-c}) < q^{-cn/r^2}.
\]
\end{theorem}
\begin{proof}
By Witt's lemma, there are $O(q^{r^2})$ orbits of $G$ on $V^r$.
Let $Gv_1, \dots, Gv_s$ be a decomposition of $V^r$ into $G$-orbits, where $s \ll q^{r^2}$.
Then
\[
  \C[V^r]_0 = \C[Gv_1]_0 \oplus \cdots \oplus \C[Gv_s]_0.
\]
Let $\rho_i = \rho(\adjcy, \C[Gv_i]_0)$ be the spectral radius of $\adjcy$ on $\C[Gv_i]_0$.
Then
\[
  \rho = \max_{1\leq i \leq s} \rho_i.
\]
From the previous subsection (possibly with a smaller $r$, if the components of $v_i$ are not linearly independent), for each $i$ we have
\[
  \P(\rho_i > q^{-c}) < q^{-c'n/r^2}.
\]
Hence
\[
  \P(\rho > q^{-c}) \ll q^{r^2 - c'n/r^2} < q^{-c'' n/r^2}.\qedhere
\]
\end{proof}

Our main interest is the conjugation action of $G$ on a conjugacy class $\CC \subset \SCl_n(q)$ of elements of degree $s = O(1)$, which is actually a quotient of an orbit of $G$ on $V^s \oplus (V^*)^s$, where $V^*$ is the dual space.
It is possible to repeat the analysis of the previous subsection allowing also $r$ factors of $V^*$,
but in fact this generalization follows formally,
since $\C[V^*] \cong \C[V]$ (as both have character $\chi(g) = q^{\dim \ker(g-1)}$),
so
\[
  \C[V^r \oplus (V^*)^r] \cong \C[V]^{\otimes r} \otimes \C[V^*]^{\otimes r} \cong \C[V]^{\otimes 2r} \cong \C[V^{2r}].
\]

\begin{corollary}
[the conjugation action on $\MM$ is expanding]
\label{cor:conjugation_is_expanding}
Let $x_1, \dots, x_k \in G$ be independent and uniformly random,
where $k > q^C$ and $n > C$.
Let $\rho = \rho(\adjcy, \C[\MM]_0)$ be the spectral radius of $\adjcy$ acting on $\C[\MM]_0$.
Then
\[
  \P(\rho > q^{-c}) \leq q^{-cn}.
\]
\end{corollary}
\begin{proof}
We claim that $\C[\MM]$ is contained in $\C[V^{2s}]$.
The map
\begin{align}
  V^s \oplus (V^*)^s
  &\to \M_n(\F_q) \\
  (v_i, \phi_i)
  &\mapsto 1 + \sum_{i=1}^s v_i \otimes \phi_i.
\end{align}
is a map of permutation representations (where $G$ acts by conjugation on $M_n(\F_q)$),
and hence induces a map of $\C[G]$-modules $\C[V^s \oplus (V^*)^s] \to \C[M_n(\F_q)]$.
The module $\C[\MM]$ is contained in the image, so it is isomorphic to a submodule of $\C[V^s \oplus (V^*)^s] \cong \C[V^{2s}]$ by complete reducibility. Hence the result follows from the previous theorem with $r = 2s$.
\end{proof}

\section{Diameter of the Cayley graph}
\label{sec:diameter}

We now collect results from the previous
sections and bound the diameter of the Cayley graph of
the subgroup of $\Cl_n(q)$ generated by random elements.

\subsection{\texorpdfstring{$\GL_n(p)$}{GL\_n(p)} and \texorpdfstring{$3$}{3} random elements}

In this subsection we prove Theorem~\ref{diameter-thm-1}. Recall that $\SL_n(p) \leq G \leq \GL_n(p)$, where $p$ is prime and $\log p < cn / \log^2 n$, the elements $x, y, z \in G$ are chosen uniformly at random, and $S = \{x^{\pm1}, y^{\pm1}, z^{\pm1}\}$. We claim that with probability $1-e^{-cn}$ we have
\begin{align}
  &\langle S \rangle \geq \SL_n(p),~\text{and} \\
  &\diam \Cay(\langle S \rangle, S) \leq n^{O(\log p)}.
\end{align}

First we show that $\langle S \rangle \geq \SL_n(p)$ with high probability. The argument is a slight modification of \cite[Section~5]{EV-SL}.\footnote{Alternatively, we could just cite \cite{kantor--lubotzky}. The given argument avoids CFSG.}

Let $\CC_1$ be the set of all irreducible $g \in \GL_n(p)$ of order $d(p^n-1)/(p-1)$ for some $d \mid (p-1)$.
Each such $g$ is equivalent to the multiplication action of some $x \in \F_{p^n}$ of the same order, and $\det g = N(x)$.
Therefore, for each $\alpha \in G^\ab \cong \F_p^\times$, the $\GL_n(p)$-classes in $\CC_{1;\alpha} = \CC_1 \cap \alpha G'$ are in bijection with elements of $\F_{p^n}$, up to Galois conjugacy, of order $d(p^n-1)/(p-1)$ and norm $\alpha$, where $d$ is the order of $\alpha$.
Note there are $\phi(d)$ elements $\alpha$ of order $d$.
Moreover, each such $g \in G$ has centralizer isomorphic to $\F_{p^n}^\times$.
Hence
\[
  \frac{|\CC_{1;\alpha}|}{|\GL_n(p)|}
  = \frac{\phi(d(p^n-1) / (p-1)) / \phi(d)}{n (p^n-1)}
  > e^{-o(n)}.
\]
Here we used the standard estimate $\phi(m) \gg m / \log \log m$.

Let $\CC_2$ be the set of all $g \in \GL_n(p)$ of order $p^{n-1}-1$ splitting $V$ as $\ell \oplus W$ for some $\ell, W$ with $\dim \ell = 1$, $\dim W = n-1$. A similar calculation shows that
\[
  \frac{|\CC_{2;\alpha}|}{|\GL_n(p)|} > e^{-o(n)}
\]
for each $\alpha \in \F_p^\times$ in this case as well. (In fact, $\CC_2$ is uniform over $\det$ fibres.)

Hence, by Corollaries~\ref{cor:expected_character_bound_exponential}
and \ref{cor:xw(y,z)-trick} as in the proof of Theorem \ref{thm:reaching_m_3},
with probability at least $1 - e^{- c n}$
there are words $w_1, w_2$ such that
\[
  w_i(x,y,z) \in \CC_i \qquad (i \in \{1, 2\}).
\]
By a straightforward adaptation of \cite[Lemma 5.2]{EV-SL} (assuming $n > 6$, say),
\[
  \langle w_1(x, y, z), w_2(x, y, z) \rangle \geq \SL_n(p).
\]
Hence indeed $\langle S \rangle \geq \SL_n(p)$.

In particular, using Schreier generators, there is a symmetric set $S' \subset S^{2p} \cap \SL_n(p)$ such that $\langle S'\rangle = \SL_n(p)$.

Meanwhile, by Theorem~\ref{thm1:finding-a-transvection}, with probability $1 - e^{-cn}$ there is another word $w$ of length $n^{O(\log p)}$ such that
\[
  w(x, y, z) \in \MM.
\]
Let $X = S' \cup \{w(x, y, z)^{\pm 1}\}$.
By \cite[Theorem~1.5]{halasi} we have
\[
  \diam \Cay(\SL_n(p), X) \ll p n^{12}.
\]
As $|\langle S \rangle/\SL_n(p)| < p$, we thus have
\[
  \diam \Cay(\langle S \rangle, S) \ll p^2 n^{12 + C \log p} = n^{O(\log p)}.
\]
This completes the proof.

\subsection{Classical groups and \texorpdfstring{$q^C$}{q\^{}C} random elements}

In this subsection we prove Theorem~\ref{diameter-thm-2}. Recall that $G = \Cl_n(q)$, where $n > C$, elements $x_1, \dots, x_k \in G$ are chosen uniformly at random where $k > q^C$, and $S = \{x_1^{\pm1}, \dots, x_k^{\pm1}\}$. We claim that with probability $1-q^{-cn}$ we have
\begin{align}
  &\langle S \rangle \geq \SCl_n(p),~\text{and} \\
  &\diam \Cay(\langle S \rangle, S) \leq q^2 n^C.
\end{align}

By Theorem~\ref{thm:reaching_m_k},
with probability at least $1 - q^{-c_1 n}$
there is a word $w$ of length at most $q^2 n^{C_1}$ so that
\[
  w(x_1, \dots, x_k) \in \MM.
\]
Let $\CC$ be the conjugacy class of $w(x_1, \dots, x_k)$ in $G$.
Note that $\CC \subseteq \SCl_n(q)$.
It follows from Corollary~\ref{cor:conjugation_is_expanding}
that, with probability at least $1 - q^{-c_2 n}$,
the conjugation action of $G$ on $\CC$ is expanding
with spectral gap bounded away from zero.
Hence
(see, e.g., \cite[Proposition~3.1.5 and Proposition~3.3.6]{kowalski2019introduction})
\[
  \diam \Sch(G, S, \CC) \ll \log |\CC|.
\]
It follows that with probability at least $1 - q^{-c_3 n}$,
every element of $\CC$ is a word in $S$ of length at most
\[
  q^2 n^{C_1} + O(\log |\CC|) \ll q^2 n^{C_2}.
\]
This already proves that $\langle S \rangle \geq \SCl_n(q)$.
It follows from \cite{liebeck2001diameters} that
\[
  \diam \Cay(\SCl_n(q), \CC) \ll \log |\SCl_n(q)| / \log |\CC| \ll n.
\]
Hence
\[
  \diam \Cay(\langle S \rangle, S)
  \ll q^2 n^{C_2 + 1}.
\]
This completes the proof.

Corollary~\ref{cor:babai}(2) follows immediately for $q < n^{O(1)}$, since $\log |G| \asymp n^2 \log q$.
If $q$ is larger then the claim follows from Alon--Roichman~\cite{alon--roichman}, which implies that the Cayley graph on $C n^2 \log q$ random generators is almost surely an expander.

\appendix

\section{Analogous arguments for \texorpdfstring{$S_n$}{S\_n}}
\label{appendix:Sn}

In this appendix we give analogous arguments for $S_n$.
The main reason to do so is to motivate and give context to some of the arguments in the main body,
as the arguments in the context of $S_n$ are easier and somewhat more natural, involving only trajectories of points rather than vectors.
A secondary reason is that a couple results are actually new, and of independent interest:
\begin{enumerate}
  \item if $w$ is a word of length $o(n^{1/2})$, then with high probability $\bar w$ has $o(n)$ fixed points (Theorem~\ref{thm:supp0.99-char1});
  \item the Cayley graph with respect to three random generators almost surely has diameter $O(n^2 \log n)$.
\end{enumerate}

\subsection{Queries and trajectories}

The following definitions only slightly generalize those in \cite{broder-shamir, FJRST}.

Let $G = S_n$ and $\Omega = \{1, \dots, n\}$.
Let $x_1, \dots, x_k \in G$.
Define a \emph{query} to be a pair $(\xi, v)$,
where $\xi \in\signedalphabet$ and $v \in \Omega$;
the \emph{result} of the query is $\bar \xi v$.
After any finite sequence of queries
\[
  (w_1, v_1), (w_2, v_2), \dots, (w_{t-1}, v_{t-1})
\]
the \emph{known domain} of a letter $\xi$ at time $t$ is
\[
  D_\xi^t = \{ v_i : w_i = \xi, i < t\} \cup \{\bar{w_i} v_i : w_i = \xi^{-1}, i < t\}.
\]
Suppose we make a further query $(w_t, v_t)$.
If $v_t \in D_{w_t}^t$, then the result $\bar {w_t} v$ is determined already by the values of $\bar{w_1} v_1, \dots, \bar{w_{t-1}} v_{t-1}$; we call this a \emph{forced choice}.
Otherwise, we say the query is a \emph{free choice}.

Let $R$ be some subset of $\Omega$ fixed in advance. If a query $(w_t, v_t)$ is a free choice and yet
\[
  \bar {w_t} v_t \in R \cup \{v_1, \bar {w_1} v_1, \dots, v_{t-1}, \bar {w_{t-1}} v_{t-1}, v_t\}
\]
then we say the result of the query is a \emph{coincidence}.

Again, the language is most interesting when $x_1, \dots, x_k \in G$ are chosen randomly. The following lemma is trivial, and parallels Lemma~\ref{lem:free-choices}.

\begin{lemma} \label{lem:free-choices-Sn}
Let $x_1, \dots, x_k \in G$ be uniformly random and independent, and let
\[
  (w_1, v_1), (w_2, v_2), \dots, (w_{t-1}, v_{t-1})
\]
be a sequence of queries.
Assume that $(w_t, v_t)$ is a free choice.
Then, conditionally on the values of $\bar{w_1} v_1, \dots, \bar{w_{t-1}} v_{t-1}$,
the result $\bar{w_t} v_t$ of the query $(w_t, v_t)$ is uniformly distributed in $\Omega \setminus D_{w_t^{-1}}^t$.

In particular, the conditional probability that $\bar {w_t} v$ is a  coincidence is bounded by
\[
  \frac{d}{n - s},
\]
where
\[
  d = | R \cup \{v_1, \bar{w_1} v_1, \dots, v_{t-1}, \bar{w_{t-1}} v_{t-1}, v_t\} |
\]
and $s$ is the number of $i < t$ with $w_i \in \{w_t, w_t^{-1}\}$.
\end{lemma}

Let $w \in F_k$, and let
\[
  w = w_\ell \cdots w_1 \qquad (w_i \in \signedalphabet)
\]
be the reduced expression.
For each $v \in \Omega$, the \emph{trajectory} of $v$ is the sequence of queries $(w_t, v^{t-1})$, where $v^0 = v$ and for each $t \geq 1$ the vector $v^t$ is the result of the query $(w_t, v^{t-1})$; in other words, the sequence $v^0, v^1, \dots, v^\ell$ is defined by
\begin{align*}
  v^0 &= v, \\
  v^t &= \bar {w_t} v^{t-1} && (1\le t\le \ell).
\end{align*}
Note that if step $t$ is free and not a coincidence then step $t+1$ is also free, and hence if $v^\ell \in R$ then there must be at least one coincidence in the trajectory (cf.~Lemma~\ref{lem:at-least-one-coincidence}).

More generally for any $r\geq1$
the \emph{joint trajectory} of an $r$-tuple $v_1, \dots, v_r \in \Omega$
is simply the $r$-tuple of individual trajectories,
with the queries $(w_t, v_i^{t-1})$ ordered lexicographically by $(t, i)$.
Again write $\prec$ for this order, i.e., $(t',i') \prec (t,i)$ if $t' < t$ or $t'=t$ and $i' < i$.
Note that if step $(t, i)$ is free and not a coincidence then
\[
  v_i^t = \bar{w_t} v_i^{t-1}
  \notin R \cup \{v_{i'}^{t'} : (t', i') \prec (t, i)\};
\]
while
\[
  D_{w_{t+1}}^{(t+1, i)}
  \subset \{v_{i'}^{t'} : (t', i') \prec (t, i)\};
\]
hence step $(t+1, i)$ is also free.
Hence if $v_i^\ell \in R$ then there must be at least one coincidence in the trajectory of $v_i$.
This observation is recorded as the following lemma (cf.~Lemma~\ref{lem:at-least-one-coincidence-joint}).

\begin{lemma} \label{lem:at-least-one-coincidence-joint-Sn}
  Suppose $v_i \notin \{v_1, \dots, v_{i-1}\}$ and $v_i^\ell \in R$. Then there is at least one coincidence in the trajectory of $v_i$ (during the joint trajectory of $v_1, \dots, v_r$).
\end{lemma}

\subsection{The probability of small support}
\label{subsec:Sn-small-support}

For $g \in S_n$, define
\[
  \fix g = \{ v \in \Omega : gv = v \}.
\]
In this section we show that if $w$ is a short word then almost surely $|\fix \bar w|$ is small.
The following lemma is similar to the argument used in \cite[Lemma~2.2]{eberhard-Sn-girth};
the only difference is that the set $R$ is fixed in advance.

\begin{lemma}
Let $G = S_n$.
Let $R \subset \Omega$ be a subset of size $r$.
Let $w \in F_k$ be a nontrivial word of length $\ell < n / r$.
Then
\[
  \P(\bar w R = R) \leq \br{\frac{\ell^2 r}{n - \ell r}}^r.
\]
\end{lemma}
\begin{proof}
Let $R = \{v_1, \dots, v_r\}$ and consider the joint trajectory of $v_1, \dots, v_r$. By Lemma~\ref{lem:at-least-one-coincidence-joint-Sn}, we can have $\bar w R = R$ only if there is at least one coincidence in each individual trajectory. We take a union bound over all possibilities for when the coincidences could occur. By Lemma~\ref{lem:free-choices-Sn}, the conditional probability that step $(t, i)$ is a coincidence is bounded by
\[
  \frac{\ell r}{n - \ell r};
\]
indeed there are at most $\ell r$ previous points (if $t = \ell$, assuming $v_j^\ell \in R$ for $j < i$).
There are $\ell^r$ possibilities for when the first coincidences might occur.
Hence the claimed bound holds.
\end{proof}

\begin{theorem} \label{thm:supp0.99-char1}
There is a constant $c>0$ such that the following holds for all $f \geq 0$.
Let $G = S_n$, and let $w \in F_k$ be a nontrivial word of reduced length $\ell < c f^{1/2}$.
Then
\[
  \P\br{|\fix \bar w| \geq f} \leq \exp\br{-c f / \ell^2}.
\]
\end{theorem}
\begin{proof}
Let $x_1, \dots, x_k$ be chosen independently and uniformly from $G$. Let $F = |\fix \bar w|$. By the lemma, for any subset $R \subset \Omega$ of size $r$ (for $r < n /\ell$) we have
\[
  \P(R \subset \fix \bar w)
  = \P(\bar w R = R)
  \leq \pfrac{r \ell^2}{n - r \ell}^r.
\]
Therefore, by a union bound,
\begin{equation} \label{eq:Sn-F-high-moment-bound}
  \E \binom{F}{r} \leq \binom{n}{r} \pfrac{r\ell^2}{n-r\ell}^r.
\end{equation}
Since $x\mapsto \binom{x}{r}$ is increasing for $x > r$, for $r < f/2$ we have
\begin{align*}
  \P\br{F \geq f}
  &\leq \binom{f}{r}^{-1} \E \binom{F}{r} \\
  &\leq \frac{n^r}{\br{f - f / 2}^r} \pfrac{r \ell^2}{n-r\ell}^r \\
  &= \br{\frac{n r \ell^2}{(f/2)(n-r\ell)}}^r
\end{align*}
Take $r \sim f / (4 \ell^2)$. The conclusion is
\[
  \P(F \geq f) \leq \exp\br{-c f / \ell^2}
\]
for some constant $c > 0$.
\end{proof}

\begin{remark}
If $\ell < c \log \log n$, a stronger bound is proved in \cite[Section~2]{larsen-shalev-fibers}.
\end{remark}

\subsection{Expected values of characters}

A notable difference between $S_n$ and $\Cl_n(q)$ is that $S_n$ has several low-degree characters:
for example, the irreducible component of the standard representation has degree $n-1$.
However, we can show that the expected value of $|\chi(\bar w)| / \chi(1)$ is smaller than $\chi(1)^{-c}$ using the Larsen--Shalev character bound~\cite{larsen--shalev-character-bound}.
For most characters, $\chi(1)$ is exponentially large in $n$, so this bound is similar in strength to Theorem~\ref{theorem:e_yz_character_bound}.
In application, low-degree characters may have to be treated specially (as in the next section).

\begin{theorem} \label{theorem:e_yz_character_bound-Sn}
Let $G = S_n$.
Let $w \in F_k$ be a fixed nontrivial word of reduced length $\ell$.
Then, for any $f \geq C \ell^2$,
\[
  \E_{x_1, \dots, x_k}
  \left( \frac{|\chi(\bar w)|}{\chi(1)} \right)
  < \exp \br{-c f / \ell^2}
  + \chi(1)^{-\frac{\log (n/f)}{2 \log n} + o(1)}.
\]
In particular, taking $f = n^{1/2}$, for $\ell < cn^{1/4}$ we have
\[
  \E_{x_1, \dots, x_k} \br{\frac{|\chi(\bar w)|}{\chi(1)}}
  < \exp (- cn^{1/2} / \ell^2) + \chi(1)^{-1/4 + o(1)}.
\]
\end{theorem}

\begin{proof}
By conditioning on whether or not $|\fix{\bar w}| \geq f$, we have
\[
  \E_{x_1, \dots, x_k}
  \br{\frac{|\chi(\bar w)|}{\chi(1)}}
  \leq
  \P_{x_1, \dots, x_k}\br{|\fix{\bar w}| \geq f}
  +
  \max_{\substack{x_1, \dots, x_k \\ |\fix{\bar w}| < f}}
  \left( \frac{|\chi(\bar w)|}{\chi(1)} \right).
\]
The first term is bounded by Theorem~\ref{thm:supp0.99-char1}. The second term is bounded by \cite[Theorem~1.3]{larsen--shalev-character-bound}.
\end{proof}

The following corollary follows exactly as in Section~\ref{sec:expected_char_values_bound}.

\begin{corollary}
  \label{cor:expected_character_bound_exponential-Sn}
There is a constant $c > 0$ such that the following holds.
Let $w$ be the result of a simple random walk of length $\ell < cn^{1/4}$ in $F_k$.
Then
\[
  \E_{x_1, \dots, x_k\in G, w}
  \left( \frac{|\chi(\bar w)|}{\chi(1)} \right)
  < \exp(-cn^{1/2} / \ell^2)
  + \chi(1)^{-1/4 + o(1)}
  + k^{-c \ell}.
\]
\end{corollary}

\subsection{Expansion in low-degree representations: a brief survey}

Let $G = S_n$, let $x_1, \dots, x_k \in G$ be random, where $k \geq 2$ and bounded, and consider the action of $x_1, \dots, x_k$ on $\Omega = \{1, \dots, n\}$. The resulting Schreier graph is one of the standard models for a random $2k$-regular graph, and the spectral properties of this graph are well studied. The earliest results on the combinatorial expansion of bounded-degree random graphs essentially coincide with the dawn of expansion, beginning with Barzdin--Kolmogorov and Pinsker (see Gromov--Guth~\cite[Section~1.2]{gromov--guth} for some history), and such results are equivalent to lower bounds on the spectral gap by the discrete Cheeger inequality (due to Dodziuk and Alon--Milman): see Kowalski~\cite[Section~4.1]{kowalski2019introduction}.

Such bounds are weak, however. The strongest results on the spectral gap of a random regular graph are based on the trace method, which is an adaptation of Wigner's proof of the semicircle law to the bounded-degree setting. These results begin with Broder and Shamir~\cite{broder-shamir}. Let $\rho$ be the spectral radius of $\adjcy$ on $\C[\Omega]_0$. Broder and Shamir proved that
\[
  \rho \ll k^{-1/4}.
\]
In particular, $\rho$ is bounded away from $1$ as long as $k$ is large enough.
On the other hand, there is a deterministic lower bound
\[
  \rho \geq (2k-1)^{1/2} / k + O(1 / \log_{2k} n),
\]
usually attributed to Alon and Boppana. The conjecture, due to Alon, that almost surely
\[
  \rho = (2k-1)^{1/2} / k + o_k(1)
\]
remained open for some time, but was finally and famously settled by Friedman, using an ingenious elaboration of the trace method: see \cite{friedman-second-eig} for the proof, and for much more background. (See also Bordenave~\cite{bordenave} for a simplified proof.)

The trace method also generalizes well, unlike the pure ``counting'' proof of expansion.
Consider the action of $\adjcy$ on $\C[\binom{\Omega}{r}]$ for bounded $r$.
This action was studied by Friedman--Joux--Roichman--Stern--Tillich~\cite{FJRST}, who showed that there is almost surely a uniform spectral gap. Their method is an elaboration of the Broder--Shamir method, and was direct inspiration for the argument of Sections~\ref{sec:closed_trajectories_one_coincidence} and \ref{sec:trace-method}. We quote their result here, which will be used in the next section:

\begin{theorem}
\label{thm:FJRST}
Let $G = S_n$, and $x_1, \dots, x_k \in G$ random.
Let $\rho = \rho(\adjcy, \C[\binom{\Omega}{r}]_0)$ be the spectral radius of $\adjcy = \adjcy_{x_1, \dots, x_k}$ acting on $\C[\binom{\Omega}{r}]_0$.
Then, for fixed $k$, $r$, and $\eps > 0$,
\[
  \P\br{\rho > (1 + \eps) (\sqrt{2k - 1} / k)^{1/(r+1)}} = o(1).
\]
\end{theorem}

\subsection{Diameter with respect to 3 random elements}

Let $G = S_n$.
Let $x_1, \dots, x_k \in G$ be random, and let $S = \{x_1^{\pm 1}, \dots, x_k^{\pm 1}\}$.
Helfgott, Seress, and Zuk~\cite{HSZ} showed that, if $k \geq 2$, then with high probability\footnote{The authors state only $n^2 (\log n)^c$, but a careful inspection of the proof gives $n^2 (\log n)^2 \omega(1)$, for an arbitrarily slowly growing $\omega(1)$. A word $v$ of length $\omega(1)$ is obtained such that $v(x, y)^{O(n)}$ has support less than $n/4$. A random commutator process is then used to iteratively reduce the support. Each step quadruples the length of the word and roughly squares the density of the support, so the whole process multiplies the length of the word by $O((\log n)^2)$. Thus a word $w$ of length $n (\log n)^2 \omega(1)$ is obtained such that $w(x, y)$ has support $3$.}
\[
  \diam \Cay(\langle S \rangle, S) \ll n^2 (\log n)^{2+o(1)}.
\]
We show in this section that if $k \geq 3$ then with high probability
\[
  \diam \Cay(\langle S \rangle, S) \ll n^2 \log n.
\]
While this is only a modest improvement, it is interesting for being conjecturally sharp for any proof which uses elements of small support as a stepping stone: it seems unlikely that an element of small support can be obtained in fewer than $O(n \log n)$ steps on average, and a generic element of $A_n$ cannot be written as a product of fewer than $O(n)$ elements of small support.

The argument is most closely related to the argument of Schlage-Puchta~\cite{schlage-puchta}, which shows that for $k = 2$ the diameter is bounded by $O(n^3 \log n)$. We get a saving for $k \geq 3$ by replacing the $xy^i$ trick with the more powerful $x w(y, z)$ trick.

\subsubsection{Alternative 1}

Write
\[
  n - 5 = n' + r
\]
where $3 \nmid n'$ and $r \in \{4, 5\}$.
Let $\CC \subset S_n$ be the normal subset of all elements whose cycle type is either $(1, 1, 3, r, n')$ or $(2, 3, r, n')$.
Note that
\[
  \frac{|\CC|}{n!} = \frac{1}{2! \cdot 3 \cdot r \cdot n'} + \frac{1}{2 \cdot 3 \cdot r \cdot n'}  \asymp 1/n,
\]
while if $\sgn$ is the sign character then
\[
  \langle 1_\CC, \sgn\rangle = \frac{(-1)^{r + n'}}{2! \cdot 3 \cdot r \cdot n'} - \frac{(-1)^{r+n'}}{2 \cdot 3 \cdot r \cdot n'} = 0.
\]

Let $x, y, z \in G$ be random.
Then by Theorem~\ref{thm:general-xw(y,z)-trick} with $f = 1_\CC$ and Corollary~\ref{cor:expected_character_bound_exponential-Sn}, if $E$ is the event that every word $u \in F_2$ of length at most $\ell < c n^{1/4}$ satisfies $x u(y, z) \notin \CC$ and $w$ is the result of a simple random walk of length $2\ell$ in $F_2$,
\begin{align}
  \P_{x,y,z}(E)
  &\ll n^2 \sum_{1 \neq \chi \in \Irr G} |\langle 1_\CC, \chi \rangle|^2
  \E_{y, z, w} \br{\frac{\chi(\bar w)}{\chi(1)}} \\
  &\leq n^2 \sum_{1 \neq \chi \in \Irr G} |\langle 1_\CC, \chi\rangle|^2
  \br{ \exp(-cn^{1/2} / \ell^2) + \chi(1)^{-1/4 + o(1)} + 2^{-c \ell} }.
\end{align}
Fixing $\ell = \floor{C \log n}$ for a sufficiently large constant $C$, we have, for sufficiently large $n$,
\begin{equation}
  \label{eqn:Sn-P(E)-char-sum}
  \P(E) \ll n^2 \sum_{1 \neq \chi \in \Irr G} |\langle 1_\CC, \chi \rangle|^2 \br{\chi(1)^{-1/5} + n^{-100}}.
\end{equation}

Let $\XX$ be the set of characters $\chi \in \Irr{G}$ such that $\chi(1) < n^{1000}$. The part of the sum \eqref{eqn:Sn-P(E)-char-sum} with $\chi \notin \XX$ is bounded by
\begin{align}
  n^2 \sum_{\chi \notin \XX} |\langle 1_\CC, \chi\rangle|^2 n^{-100} \ll n^{-98} \sum_{\chi \in \Irr{G}} |\langle 1_\CC, \chi\rangle|^2 \asymp n^{-99}.
\end{align}
Now consider some $\chi \in \XX$. Let $\pi \in \CC$. It follows from the Murnaghan--Nakayama rule (splitting off an $n'$-cycle) that $|\chi(\pi)| = O(1)$. Hence
\[
  |\langle 1_\CC, \chi\rangle | \ll \frac{|\CC|}{|G|} \asymp n^{-1}.
\]
It follows from the hook length formula that $|\XX| = O(1)$.
Hence, since $\langle 1_\CC, \sgn\rangle = 0$,
\[
  n^2 \sum_{1 \neq \chi \in \XX} |\langle 1_\CC, \chi \rangle|^2
  (\chi(1)^{-1/5} + n^{-100})
  \ll n^{-1/5}
\]
(the main term coming from the characters of degree $n-1$).
Hence, from \eqref{eqn:Sn-P(E)-char-sum},
\[
  \P(E) \ll n^{-1/5}.
\]

We conclude that with high probability there is a word $w \in F_3$ of length $O(\log n)$ such that $w(x, y, z) \in \CC$. Hence there is a word $w' = w^{2 r n'}$ of length $O(n \log n)$ such that $w'(x, y, z)$ is a $3$-cycle. With high probability the conjugation action of $x, y, z$ on the set of $3$-cycles has a uniform spectral gap (by Theorem~\ref{thm:FJRST}), so it follows that every $3$-cycle is a word in $x, y, z$ of length $O(n \log n)$. Thus every element of $A_n$ is a word in $x, y, z$ of length $O(n^2 \log n)$.

\subsubsection{Alternative 2}

The crude bound $n^{-1/5}$ for the probability can be improved as follows.
Write
\[
  n - 101 = n' + r
\]
where $101 \nmid n'$ and $r \in \{99, 100\}$.
Let $\CC \subset S_n$ be the normal subset of all elements having both a 101-cycle and an $n'$-cycle (the remaining part is an arbitrary element of $S_r$).
Assuming $n' > 101$,
\[
  \frac{|\CC|}{n!} = \frac{1}{101 n'} \asymp 1/n,
\]
and as before we have $\langle 1_\CC, \sgn \rangle = 0$. In fact, $\langle 1_\CC, \chi \rangle = 0$ for all low-degree $\chi$.

\begin{lemma}
If $1 \neq \chi \in \Irr G$ and $\langle 1_\CC, \chi \rangle \neq 0$, then $\chi(1) \gg n^{98}$.
\end{lemma}
\begin{proof}
It is well-known that characters of $S_n$ are parameterized by partitions $\lambda \vdash n$.
Let $\chi = \chi_\lambda$ be a character such that $\langle 1_\CC, \chi \rangle \neq 0$.
By the Murnaghan--Nakayama rule, it must be the case that $\lambda$ can be obtained by starting from $(r)$ and adding a $101$-rim-hook and an $n'$-rim-hook.
Hence if $\chi$ is nontrivial and $n$ is sufficiently large then $\lambda_1 \leq n - 100$ and $\lambda_1' \leq n - 98$.
From the hook length formula it follows that, for sufficiently large $n$,
\[
  \chi(1) \geq \chi_{(99, 1^{n-99})}(1) = \frac{n!}{n \, 98! \, (n-99)!} \asymp n^{98}.\qedhere
\]
\end{proof}

It follows as before that, with probability at least
\[
  1 - O(n^{-98/5}),
\]
there is a word $w \in F_3$ of length $O(\log n)$ such that $w(x, y, z) \in \CC$. Hence there is a word $w' = w^{r! n'}$ of length $O(n \log n)$ such that $w'(x, y, z)$ is a $101$-cycle. By Theorem~\ref{thm:FJRST} (and inspecting the proof), the conjugation action of $x, y, z$ on the set of $101$-cycles has spectral gap at least $\delta$ with probability at least
\[
  1 - O(n^{-1 + O(\delta) + o(1)}).
\]
Taking $\delta = 1/\log n$ (say), it follows that every $101$-cycle is a word in $x, y, z$ of length $O(n \log n)$,
and hence the diameter of $\Cay(\langle S \rangle, S)$ is $O(n^2 \log n)$, with probability
\[
  1 - n^{-1+o(1)}.
\]

\bibliography{refs}
\bibliographystyle{alpha}
\end{document}